\documentclass[11pt]{amsart}
\usepackage{amssymb,amsmath,amsthm,amscd}

\usepackage{graphicx}
\numberwithin{equation}{section}

\newtheoremstyle{fancy1}{10pt}{10pt}{\itshape}{12pt}{\textsc\bgroup}{.\egroup}{8pt}{
}
\newtheoremstyle{fancy2}{10pt}{10pt}{}{12pt}{\itshape}{.}{8pt}{ }

\theoremstyle{fancy1}

\newtheorem{cor}[equation]{Corollary}
\newtheorem{lem}[equation]{Lemma}
\newtheorem{prop}[equation]{Proposition}

\newtheorem*{thm*}{Theorem}

\newtheorem*{main*}{Theorem}

\newtheorem*{cor*}{Corollary}
\newtheorem*{prop*}{Proposition}

\newtheorem*{problem*}{Problem}

\theoremstyle{fancy2}

\newtheorem{rem}[equation]{Remark}
\newtheorem*{rems*}{Remarks}

\newtheorem*{rem*}{Remark}
\newtheorem{example}{Example}
\newtheorem*{example*}{Example}

\newcommand{\cref}[1]{Corollary~\ref{#1}}

\newcommand{\pref}[1]{Proposition~\ref{#1}}

\newcommand{\sref}[1]{Section~\ref{#1}}

\newcommand{\e}{\epsilon}

\newcommand{\M}{\mathcal M}
\newcommand{\DD}{\mathcal D}

\newcommand{\Sph}{\mathbb{S}}

\newcommand{\C}{{\mathbb{C}}}
\newcommand{\R}{{\mathbb{R}}}

\newcommand{\N}{{\mathbb{N}}}

\newcommand{\fg}{{\mathfrak{g}}}
\newcommand{\fk}{{\mathfrak{k}}}
\newcommand{\fh}{{\mathfrak{h}}}
\newcommand{\fm}{{\mathfrak{m}}}

\newcommand{\fsu}{{\mathfrak{su}}}

\newcommand{\fz}{{\mathfrak{z}}}

\def\con#1=#2(#3){#1 \equiv #2 \bmod{#3}}

\newcommand{\tr}{\ensuremath{\operatorname{tr}}}
\newcommand{\diag}{\ensuremath{\operatorname{diag}}}
\newcommand{\Aut}{\ensuremath{\operatorname{Aut}}}

\newcommand{\rank}{\ensuremath{\operatorname{rk}}}

\newcommand{\Ad}{\ensuremath{\operatorname{Ad}}}

\renewcommand{\sec}{\ensuremath{\operatorname{sec}}}
\newcommand{\Ric}{\ensuremath{\operatorname{Ric}}}

\DeclareMathOperator{\spam}{span}
\DeclareMathOperator{\grad}{grad}

\newcommand{\bi}{\mathbf{i}}
\newcommand{\bj}{\mathbf{j}}
\newcommand{\bk}{\mathbf{k}}

\begin{document}

\title{On the variational properties of the prescribed Ricci curvature functional}

\author{Artem Pulemotov}
\address{University of Queensland}
\email{a.pulemotov@uq.edu.au}
\author{Wolfgang Ziller}
\address{University of Pennsylvania}
\email{wziller@math.upenn.edu}

\thanks{This research was supported by the Australian Government through the Australian Research Council's Discovery Projects funding scheme (project DP180102185).  The second author was also supported by the NSF  grant 1506148 and a Reybould Fellowship.}

\begin{abstract}
We study the prescribed Ricci curvature problem for homogeneous metrics. Given a (0,2)-tensor field~$T$, this problem asks for solutions to the equation $\Ric(g)=cT$ for some constant~$c$. Our approach is based on examining global properties of the scalar curvature functional whose critical points are solutions to this equation. We produce conditions for a general homogeneous space under which it has a global maximum.  Finally, we study the behavior of the functional in specific examples to illustrate our result.
\end{abstract}

\maketitle

The prescribed Ricci curvature problem consists in finding a Riemannian metric $g$ on a manifold $M$ such that $\Ric(g)=T$ for a given (0,2)-tensor field~$T$. As was suggested by DeTurck~\cite{D85} and Hamilton~\cite{RH84}, it is more natural to ask instead whether one can solve the equation
\begin{align*}
\Ric(g)=cT
\end{align*}
for some constant~$c$. In fact, on a compact manifold, such a constant appears to be necessary. For example, if $M$ is a surface, this follows from the Gauss--Bonnet theorem.

\smallskip

The prescribed Ricci curvature problem has been studied by many authors since the 1980s; see, e.g.,~\cite{BP19} for an overview of the literature. Considering the difficulty of the equation involved, it is natural to make symmetry assumptions. More precisely, suppose that the metric $g$ and the tensor $T$ are invariant under a Lie group $G$ acting on~$M$. In the case where the quotient $M/G$ is one-dimensional, the problem was addressed by Hamilton~\cite{RH84}, Cao--DeTurck~\cite{CD94}, Pulemotov~\cite{AP13a,APadd} and Buttsworth--Krishnan~\cite{BK20}. The case where $M$ is a homogeneous space~$G/H$ has been studied extensively; see the survey~\cite{BP19} and the more recent references~\cite{BPRZ21,LWa,LWb,APZ21,AGP21}. In some simple situations, the equation can be solved explicitly, as shown, e.g., in~\cite{AP16,TB16}.
 
\smallskip
 
Throughout the paper we assume that $T$ is positive-definite. Without this assumption, the behavior of $S_{|\M_T}$ is very different.  General existence theorems in the homogeneous setting rely on the fact, proven in~\cite{AP16}, that $G$-invariant metrics on $G/H$ with Ricci curvature $cT$ are precisely (up to scaling) the critical points of the scalar curvature functional $S$ on the set  $\M_T=\M_T(G/H)$ of $G$-invariant metrics on $G/H$ subject to the constraint $\tr_gT=1$. We will assume that $G$ and $H$ are compact Lie groups. The goal of this paper is to study the global behavior of the scalar curvature functional.  In \cite{MGAP18,AP19} one finds a general theorem for the existence of a global maximum, under certain assumptions on the isotropy representation of~$G/H$. Our main result removes these assumptions, thereby expanding greatly the class of spaces to which the theorem applies. Moreover, we improve the conditions for the existence of the global maximum and interpret them geometrically.
 
\smallskip
 
If $H$ is maximal in $G$, the supremum of $S$ is always attained. Otherwise, the global behavior of $S$ depends on the set of intermediate subgroups, i.e., Lie groups $K$ with $H\subset K\subset G$. For every such $K$, one has the homogeneous fibration
\begin{equation*}
K/H\to  G/H \to G/K.
\end{equation*}
If we fix homogeneous metrics $g_F$ and $g_B$ on the fiber $K/H$ and the base $G/K$, it is natural to study the two-parameter variation 
\begin{equation*}
g_{s,t}=\frac1s \,g_F + \frac1t\,g_B
\end{equation*}
of $g=g_F+g_B$. Notice that we define this variation using the reciprocals of $s$ and~$t$. In fact, we often deal with inverses of metrics rather than metrics themselves, which enables us to view $\M_T$ as a pre-compact space. Solving  the constraint $\tr_gT=1$ for~$s$, we obtain a one-parameter family $g_t\in S_{|\M_T} $, called a {\it canonical variation}, with scalar curvature given by 
\begin{equation*}
S(g_t)= \frac{S_F}{T_1^*}+T_2^*\Big(\frac{S_B}{T_2^*}-\frac{S_F}{T_1^*}\Big)t-\frac{t^2T_1^*}{1-tT_2^*}|A|_g.
\end{equation*}
Here $T_1^*=\tr_{g_F}T_{|F}$ and $T_2^*=\tr_{g_B}T_{|B}$ are the traces of $T$ on the fiber and the base, $S_F$ and $S_B$ are the scalar curvatures of $g_F$ and $g_B$, and $A$ is the O'Neil tensor of the Riemannian submersion. As a consequence,
\begin{equation*}
\lim_{t\to0}S(g_t)=\frac{S_F}{T_1^*} \qquad\text{and}\qquad
\lim_{t\to0}\frac{dS(g_t)}{d t}=T_2^*\Big(\frac{S_B}{T_2^*}-\frac{S_F}{T_1^*}\Big),
\end{equation*}
which means that $S_F$ and $S_B-S_F$ control the behavior of $S$ at infinity. It is easy to see that $S$ is bounded from above if $T$ is positive-definite, and it is therefore natural to search for a global maximum of~$S$. The above formulas motivate two invariants of the intermediate subgroup $K$ with Lie algebra~$\fk$:
\begin{align*}
\alpha_\fk=\sup\Big\{\frac{S(h)}{\tr_hT_{|F}}\,\Big|\, h\in \M_T(K/H)\Big\}\qquad\text{and}\qquad  \beta_\fk=\sup\Big\{\frac{S(h)}{\tr_hT_{|B}} \,\Big|\, h\in \M_T(G/K) \Big\}.
\end{align*}
These quantities maximize the limits and the ``derivatives" of $S$ at infinity. We furthermore introduce an invariant of the homogeneous space $G/H$:
\begin{equation*}
\alpha_{G/H}=\textstyle\sup_{\fk}\alpha_\fk,
\end{equation*}
where the supremum is taken over all Lie algebras of intermediate subgroups. We can now state our main result.
 
\begin{main*}{\label{main_1}} Let $G/H$ be a compact homogeneous space  and $T$ a positive-definite $G$-invariant (0,2)-tensor field on $G/H$.	If $H$ is not maximal in $G$, then the set of intermediate subgroups $K$ with $\alpha_\fk=\alpha_{G/H}$ is non-empty. If $K$ is such a subgroup of the lowest possible dimension and $\beta_{\fk}-\alpha_{\fk}>0$, then $S_{|\M_T}$ achieves its supremum at some metric~$g\in\M_T$ and hence $\Ric(g)=cT$ for some $c>0$.
\end{main*}

In fact, we show that there exists $\e>0$ such that the set $\{ g\in \M_T(G/H)\mid S(g)>\alpha_{G/H}+\e \}$ is non-empty and compact, which implies the result. This requires careful estimates of the behavior of~$S_{|\M_T}$ at infinity.

\smallskip

We illustrate this result by examining as examples the Stiefel manifold $V_2(\R^5)$ and Ledger--Obata spaces. These manifolds are especially interesting since the isotropy representation has equivalent modules and there are families of intermediate subgroups. Hence not every metric is ``diagonal" (see discussion below) as was assumed in all previously considered examples. The Euler--Lagrange equations for $S_{|\mathcal M_T}$ are complicated and cannot be solved explicitly. Our theorem implies the existence of a large set of tensors $T$ such that  $S_{|\M_T}$ has a global maximum, although in many cases, $S_{|\M_T}$ has no critical points at all. These examples also have interesting features for critical points besides global maxima. We find:
\smallskip
\begin{itemize}
	\item[\small{$\blacktriangleright$}]
	A tensor $T$ such that $S_{|\M_T}$ admits a two-dimensional submanifold of critical points. This submanifold is non-degenerate and a local maximum.
	
	\item[\small{$\blacktriangleright$}]
	A large set of critical points that are global maxima among diagonal metrics but saddles in the space of all invariant metrics.
	
	\item[\small{$\blacktriangleright$}]
	A set of tensors $T$ for which $S_{|\M_T}$ has both a circle of saddles and a strict local maximum, although for some of these $T$, the functional does not achieve its global maximum.
\end{itemize}
\smallskip

We note that generic critical points are isolated. More precisely, as shown in~\cite{LWa}, there exists an open and dense subset in the space of $G$-invariant metrics on which the Ricci map is a local diffeomorphism (up to scaling).
\smallskip

Finally, we observe that the functional $S_{|\M_T}$ has a large set of critical points on flag manifolds, e.g., $SU(3)/T^2$. Let $J$ be a complex structure  on a compact generalized flag manifold $G/H=G/C(\tau)$, where $C(\tau)$ is the centralizer of a torus $\tau$ in~$G$. It is well known that there exists a unique $G$-invariant Hermitian  bi-linear form $h_J$ such that $\Ric(g) = -h_J$ for every K\"ahler metric $g$ compatible with~$J$; see~\cite{K55}.  Since the set of such metrics has the same dimension as the torus~$\tau$, we obtain a smooth critical submanifold in $\M_T$ of dimension $\dim\tau-1$. Furthermore, this critical submanifold contains the unique K\"ahler--Einstein metric associated to~$J$.  For instance, if $G/H = SU(n + 1)/T^n$, there is an $(n-1)$-dimensional smooth submanifold of K\"ahler metrics consisting of critical points of $S_{|\M_T}$ for $T$ equal to~$-h_J$ (up to scaling). In general, there exist several inequivalent complex structures on $G/H$ and hence non-isometric critical submanifolds.
\smallskip

We now outline the strategy of our proofs. To describe the set of homogeneous metrics on $G/H$, one decomposes the tangent space $\fm\simeq\fg/\fh$ into a sum of irreducible modules $\fm_1\oplus\cdots\oplus\fm_r$ and considers so-called diagonal metrics 
$$
g=x_1Q_{|\fm_1}+ x_2Q_{|\fm_2} +\cdots+  x_r Q_{|\fm_r},
$$
where $Q$ is a fixed bi-invariant metric on the Lie algebra of~$G$. Not every homogeneous metric is of this form when some of the summands $\fm_i$ are equivalent. However, it can always be reduced to this form by choosing another decomposition of $\fm$ as above. With the substitution $y_i=\frac{1}{x_i}$, the constraint $\tr_gT=1$ becomes simply $\sum d_iT_iy_i=1$, where $d_i=\dim\fm_i$ and $T_{|\fm_i}=T_i\, Q_{|\fm_i}$. Thus the set of diagonal metrics in $\M_T$ is parameterized by a simplex $\Delta$ with stratified boundary~$\partial\Delta$. In this construction, the strata are indexed by the sets of those variables $y_i$ that vanish. Some of them are marked by subalgebras $\fk$ with $\fh\subset\fk\subset\fg$ and denoted~$\Delta_\fk$. If $(g_i)\subset\M_T$ is a divergent sequence of metrics of bounded scalar curvature, then there exist such an intermediate subalgebra and a subsequence of $(g_i)$ that converges to a point in~$\Delta_\fk$. At the remaining strata, denoted~$\Delta_\infty$, the scalar curvature goes to $-\infty$. We will in fact show that $\alpha_\fk$ is the largest possible limit of the scalar curvature achieved as one approaches~$\Delta_\fk$. Since $S_{|\M_T}$ is bounded from above, it is natural to look at subalgebras $\fk$ with $\alpha_\fk=\alpha_{G/H}$ and determine conditions under which there are metrics with scalar curvature larger than $\alpha_{G/H}$. This is achieved by using the ``derivatives" $\beta_\fk-\alpha_\fk$, and if such metrics exists, we show that $\{ g\in \M_T(G/H)\mid S(g)>\alpha_{G/H}+\e \}$ is compact for some $\e>0$. One of the main difficulties is that, generally speaking, the scalar curvature does not extend continuously to $\partial\Delta$ and $\alpha_\fk$ may not be achieved by a metric on~$K/H$. In addition, $\beta_\fk-\alpha_\fk$ may not be an actual derivative. This also explains, in part, why we cannot arbitrarily choose a subalgebra $\fk$ with $\alpha_\fk=\alpha_{G/H}$ in our main theorem. For the proof, we need to produce careful estimates for the scalar curvature near $\partial\Delta$. The final difficulty is to extend our arguments to the set of all metrics in $\M_T$ when some of the isotropy summands are equivalent and $\fm$ can be decomposed into irreducible modules in many substantially different ways. This requires further careful estimates.

\smallskip

The paper is organized as follows. In  \sref{prelim} we recall properties of homogeneous spaces and Riemannian submersion that we will need in our proofs. In \sref{sec_complex} we describe the simplicial complex $\Delta$ and the stratification of $\partial\Delta$ by subalgebras, as well as the behavior of the scalar curvature near the strata in $\partial\Delta$. In \sref{sec_submersions} we use  Riemannian submersions defined by intermediate subgroups to understand when there exist metrics $g$ with $S(g)>\alpha_\fk$ near a stratum $\Delta_\fk$. In \sref{sec_Maxima} we vary the decompositions and prove our main theorem.  Finally, in \sref{sec_Example} we study the Stiefel manifold and the Ledger--Obata space as examples.
 
\bigskip

\section{Preliminaries}\label{prelim}

\smallskip

We first recall some basics of the geometry of a homogeneous space. Let $H\subset G$ be two compact Lie groups with Lie algebras~$\fh\subset\fg$ such that  $G/H$ is an almost effective homogeneous space. 
We fix a bi-invariant metric $Q$ on~$\fg$, which defines a $Q$-orthogonal $\Ad_H$-invariant splitting $\fg=\fh\oplus\fm$, i.e., $\fm=\fh^\perp$. The tangent space $T_{eH} (G/H)$ is identified with $\fm$, and $H$ acts on $\fm$ via the adjoint representation $\Ad_H$.  A $G$-invariant metric on $G/H$ is determined by an $\Ad_H$-invariant inner product on~$\fm$. We denote by $\M(G/H)$, or sometimes simply by~$\M$, the space of $G$-invariant metrics on $G/H$. We will  assume throughout the paper that $G/H$ is not a torus since in this case all  $G$-invariant metrics are flat, in particular, $\Ric(g)=0$ for all~$g\in\M$.
 
We describe metrics in $\M$ in terms of $\Ad_H$-invariant splittings. Under the action of $\Ad_H$ on $\fm$, we decompose
\begin{equation*}
\fm=\fm_1\oplus\ldots\oplus\fm_r,
\end{equation*}
where $\Ad_H$ acts irreducibly on $\fm_i$. Some of these summands may need to  be one-dimensional if there exists a subspace of $\fm$ on which $\Ad_H$ acts as the identity. We denote by $\DD$ the space of all such decompositions and use the letter $D\in\DD$ for a particular choice of decomposition. The space $\DD$ has a natural topology induced from the embedding into the products of Grassmannians $G_k(\fg)$ of $k$-planes in $\fg$. Clearly, $\DD$ is compact.

If $T$ is a $G$-invariant symmetric bi-linear form field on~$G/H$, it is determined by its value on~$\fm$. We are interested when such a bi-linear form field is (up to scaling) the Ricci curvature of a metric~$g\in\M$, i.e., when
\begin{equation}\label{prescribe}
\Ric(g)=cT \qquad \text{for some constant $c$.}
\end{equation}
Throughout the paper we will assume that $T$ is positive-definite. We may also assume that $c>0$ in~\eqref{prescribe} since a compact homogeneous space does not admit any metrics with $\Ric\le 0$ unless it is a torus (see~\cite[Theorem~1.84]{AB87}), which we excluded above.

Define the hypersurface
\begin{equation*}
\M_T(G/H)=\{g\in\M\mid \tr_gT=1\}\subset\M,
\end{equation*}
where $\tr_g$ is the trace with respect to $g$. We denote it simply by $\M_T$ when the homogeneous space is clear from context. As shown in~\cite{AP16}, a solution to~\eqref{prescribe} can we viewed as a critical point of the functional
\begin{equation*}
S\colon \M_T \to \R,
\end{equation*}
where $S(g)$ is the scalar curvature of $g$. More precisely, the following result holds.

\begin{prop}\label{critical}
	The Ricci curvature of a metric $g\in\mathcal M_T$ equals $cT$ for some $c\in\mathbb R$ if and only if $g$ is a critical point of $S_{|\mathcal M_T}$.
\end{prop}

Here $c$ is the Lagrange multiplier of the variational problem. Our main interest in this paper is to describe the geometry of the functional $S$ and its implications for when~\eqref{prescribe} has a solution. 

We now recall the formulas for the scalar curvature and the Ricci curvature of a homogeneous metric. Given $g\in\M$, we have  $Q_{|\fm_i}=x_i\, Q_{|\fm_i}$ for some constant $x_i>0$. In general, $\fm_i$ and $\fm_j$ do not have to be orthogonal if some of these summands are equivalent. But we can diagonalize $g$ and $Q$ simultaneously, and hence there exists a decomposition $D\in\DD$ such that the metric has the form 
\begin{equation}\label{diagonalmetric}
g=x_1Q_{|\fm_1}+ x_2Q_{|\fm_2} +\cdots+  x_r Q_{|\fm_r}.
\end{equation}
We call such metrics diagonal with respect to our choice of $D$ and denote their set by $\M^D(G/H)$, or simply $\M^D$. Thus, $\M=\cup_{D\in\DD}\M^D$. We also denote $\M_T^D=\M_T\cap\M^D$. The scalars $x_i$ are simply the eigenvalues of $g$ with respect to~$Q$. When these eigenvalues have multiplicity, and some of the modules in the corresponding eigenspace are equivalent under the action of $\Ad_H$, we note that $g\in\M^D$ for a compact infinite family of decompositions~$D$.  In order to describe all homogeneous metrics on~$G/H$, we can thus  restrict ourselves to diagonal metrics but allow the decomposition to change. The advantage is that, while the scalar curvature of a homogeneous metric for a fixed decomposition is quite complicated and hence somewhat intractable, for a diagonal metric it has a much simpler form. This idea was first used in~\cite{MWWZ86} to study $G$-invariant Einstein metrics.

We define the structure constants
\begin{equation*}
[ijk]=\sum_{\alpha,\beta,\gamma} Q([e_\alpha,e_\beta],e_\gamma)^2,\qquad i,j,k=1,\cdots,r,
\end{equation*}
 where $(e_\alpha)$, $(e_\beta)$ and $(e_\gamma)$ are $Q$-orthonormal bases of $\fm_i$, $\fm_j$ and~$\fm_k$. Clearly, $[ijk]\ge 0$, and $[ijk]=0$ if and only if $Q([\fm_i,\fm_j],\fm_k)=0$. We will denote by $B$  the Killing form of~$G$. By the irreducibility of~$\fm_i$, there exist constants $b_i\ge0$ such that 
\begin{equation*}
B_{|\fm_i}=-b_i Q_{|\fm_i} 
\end{equation*}
with $b_i=0$ if and only if $\fm_i$ lies in the center of $\fg$. Furthermore,  not all of $b_i$ vanish since otherwise $\fm$ is in the center $\fz(\fg)$ and $G/H$ is a torus.
Using this notation, and $d_i=\dim\fm_i$, the scalar curvature is given by
\begin{align}\label{scalcurvx}
S(g)&=\frac12\sum_{i}\frac{d_ib_i}{x_i}-\frac14\sum_{i,j,k}[ijk]\frac{x_k}{x_ix_j}
\end{align}
(see~\cite{MWWZ86}).
For the tensor $T$, we introduce the constants $T_i$ such that
\begin{equation*}
T_{|\fm_i}=T_i\, Q_{|\fm_i}.
\end{equation*}
Varying over all decompositions, these constants determine~$T$ uniquely.

We now recall some formulas for Riemannian submersions which will be useful for us. Let $K$ be a compact subgroup with Lie algebra $\fk$ lying between $H$ and~$G$, i.e., $H\subset K\subset G$. We then have a homogeneous fibration
\begin{equation*}
K/H\to  G/H \to G/K.
\end{equation*}
We will often consider metrics with respect to which the projection $G/H\to G/K$ is a Riemannian submersion. Assume that the decomposition $\fm=(\fk\cap\fm)\oplus \fk^\perp$ is orthogonal with respect to both $Q$ and~$g$. Then $g\in\M$ is a Riemannian submersion metric if and only if $g_{|\fk^\perp}$ is $\Ad_K$-invariant. In this case, $g_{|\fk\cap\fm}$ can be thought of as a homogeneous metric on the fiber $F=K/H$, and $g_{|\fk^\perp}$ a homogeneous metric on the base $B=G/K$.  We can introduce a new submersion metric $g_{s,t}$ on $G/H$ by scaling the fiber and the base, i.e., 
\begin{equation*}
g_{s,t}=\frac1s \,g_F + \frac1t\,g_B, \qquad \text{ where } \qquad g_F=g_{|\fk\cap\fm}\qquad \text{and} \qquad g_B=g_{|\fk^\perp}.
\end{equation*}
One has the following formula for the scalar curvature (see, e.g.,~\cite[Proposition~9.70]{AB87}):
\begin{equation}\label{scalar_sub}
S(g_{s,t})= s S_F + t S_B-\frac{t^2}{s}|A|_g,
\end{equation}
where $A$ is the O'Neill tensor of the submersion, while $S_F$ and $S_B$ are the scalar curvatures of $g_F$ and~$g_B$.
The proof of this formula is a pointwise calculation and hence extends to the situation where $K$ is not compact (i.e., $G/K$ is not a manifold). In our case it can indeed happen that for some Lie algebra $\fk\subset\fg$ the connected subgroup $K\subset G$ with Lie algebra $\fk$ is not compact. However, this will not  affect our discussions.

Let $N(H)$ denote the normalizer of $H$. Every element $n\in N(H)$ acts on $G/H$ as a diffeomorphism via right translation:  $R_n(gH)= gn^{-1}H$. Combining this with the action by left translation, we obtain an action on the space of metrics via pullback: $\M\ni g\mapsto (\Ad_n)^*(g)\in\M$. This induces an isometry between $g$ and $(\Ad_n)^*(g)$, and hence $S((\Ad_n)^*(g))=S(g)$ for the scalar curvature. This also holds for the possibly larger group $\Aut(G,H)$ of automorphisms of $G$ that preserve $H$. Note though that $S_{|\M_T}$ is invariant under $\Aut(G,H)$, or one of its subgroups, only if $T$ is as well, in which case the orbit of a critical point consists of further critical points.

The following remark, based on Palais's principle of symmetric criticality, will be useful for us. If $T$ is invariant under a subgroup $L\subset \Aut(G,H)$, and  if $\M_T^L$ is the set metrics in $\M_T$ invariant under~$L$, then critical points of the restriction $S_{|\M_T^L }$
are also critical points of $S_{|\M_T}$. Indeed,  given $g\in \M_T^L$, the Ricci curvature $\Ric(g)$ and hence the gradient $\grad S_{|\M_T}(g)$ is invariant under~$L$. Consequently, $\grad S_{|\M_T}(g)$ must be tangent to~$\M_T^L$.

In the remainder of the paper we will assume for simplicity that $G$, $H$ and all the intermediate subgroups are connected. Let us explain why this is in fact not necessary. One only needs to make the following modification. 
If $G$ or $H$ is not connected, we consider only intermediate subalgebras $\fk$ that are Lie algebras of intermediates subgroups $H\subset K\subset G$. This is easily seen to be  equivalent to saying that $\fk$ must be invariant under~$\Ad_H$.
The proofs of all of our results apply without any changes, and the conclusions are also the same. This is useful when the space of invariants has large dimension or there are many intermediate subgroups. Adding components to $G$ and $H$ will reduce $\dim\M_T$ and may easily imply the existence of some critical points. This happens, for instance, when $G/H$ is isotropy irreducible; see~\cite{WZ91} for many examples.

\bigskip
\section{The simplicial complex}\label{sec_complex}

\smallskip

In Sections~\ref{sec_complex}--\ref{sec_submersions} we fix the decomposition $D\in\DD$ and study the set of metrics $\M^D$ diagonal with respect to~$D$.

It will be convenient for us to describe a homogeneous metric in terms of its inverse since this makes the space of metrics precompact. If $g\in\M_T^D$ is given by~\eqref{diagonalmetric}, we set $y_i=\frac1{x_i}$ and obtain the following formulas for the scalar curvature and its constraint:
\begin{equation}\label{scalar_y}
S(g)=\frac12\sum_{i}d_ib_iy_i-\frac14\sum_{i,j,k}\frac{y_iy_j}{y_k}[ijk],
\qquad \tr_gT=\sum_{i}d_iT_iy_i=1.
\end{equation}
We need to study the behavior of $S_{|\M_T^D}$ at infinity, which means that at least one of the variables $y_i$ goes to $0$. It is natural to introduce a simplicial complex and its stratification. Specifically, let 
\begin{equation*}
\Delta=\Delta^D=\big\{(y_1,\cdots, y_r)\in \R^r\,\big|\, \textstyle\sum_{i} d_iT_iy_i=1 \text{ and } y_i> 0\big\}.
\end{equation*}
Notice that the numbers $T_i$, and thus the simplex~$\Delta$, depend on the choice of~$D$. This simplex is a natural parametrization of the set~$\M_T^D$. We identify a metric $g\in\M_T^D$ with $y=(y_1,\cdots, y_r)\in\Delta$.

The boundary of $\Delta$ consists of lower-dimensional simplices. For every nonempty proper subset $J$ of the index set $I=\{1,\cdots,r\}$, let
\begin{equation*}
\Delta_J=\{y\in\partial\Delta\mid y_i>0 \text{ for } i\in J,\ y_i=0 \text{  for } i\in J^c\}.
\end{equation*}
Thus $\Delta_J$ is a $|J|$-dimensional simplex, which we call a \emph{stratum} of~$\partial\Delta$.  
The closure of $\Delta_J$ satisfies
\begin{equation*}
\bar{\Delta}_J=\bigcup_{J'\subset J}\Delta_{J'},
\end{equation*}
and we call $\Delta_{J'}$ a stratum adjacent to $\Delta_{J}$ if $J'$ is a nonempty proper subset of~$J$. It will also be useful for us to consider tubular $\e$-neighborhoods of strata for $\e>0$:
\begin{equation*}
T_\e(\Delta_J)=\{y\in\Delta \mid y_i\le \e \text{ for }  i\in J^c\}.
\end{equation*}
Finally, we associate to each stratum an $\Ad_H$-invariant subspace of $\fm$:
\begin{equation*}
\fm_J=\bigoplus_{i\in J}\fm_i.
\end{equation*}

We can fill out the closure $\bar\Delta$ with geodesics starting at the center. To this end, consider the unit sphere
$$
\Sph=\big\{v\in\R^r \,\big|\, \textstyle\sum_i d_iT_iv_i=0,~\sum v_i^2=1\big\}
$$
of dimension $r-2$. Define a geodesic $\gamma_v\colon [0,t_v]\to\bar\Delta$ by setting $\gamma_v(t)=v_0-tv$, where
\begin{equation*}
v\in\Sph,\qquad t_v=\frac1{r\max_id_iT_iv_i},\qquad \text{ and }\qquad v_0=(v_{01},\ldots,v_{0r})=\Big(\frac1{rd_1T_1},\ldots,\frac1{rd_rT_r}\Big).
\end{equation*}
The stratification of $\partial\Delta$ induces one of the sphere:
\begin{equation*}
\Sph_J=\{v\in\Sph\mid \gamma_v(t_v)\in \Delta_J\}.
\end{equation*}

Our first observation is that we can mark the strata with subalgebras.

\begin{prop}\label{infinity}
The functional $S_{|\M_T^D}$ is bounded from above. Furthermore, for any $v\in\Sph$ either $S(\gamma_v(t))\to-\infty$ as $t\to t_v$ or $v\in\Sph_J$ for some $J$ such that $\fh\oplus\fm_J$ is a subalgebra of~$\fg$. 
\end{prop}

\begin{proof} Let $A=\frac{\max_ib_i}{\min_i T_i}$, which is well-defined  since $T_i>0$ by assumption. We also have $A>0$ since $G/H$ is not a torus. Then \eqref{scalar_y} implies that $S(g)\le A$.
	
For the second claim,  let   $J$ be the index set with $J^c=\{i\mid v_0-t_vv_i=0\}$. Obviously, $J\ne I$ and $t_v>0$. 
If $\fh\oplus\fm_J$ is not a subalgebra, then there exist $i,j\in J$ and $\ k\in J^c$ such that $[ijk]\ne 0$. Then in formula~\eqref{scalar_y}, we have a contribution of the form
	\begin{equation*}
-[ijk]\,\frac{y_iy_j}{y_k}	= -[ijk]\,\frac{(v_{0i}-tv_i)(v_{0j}-tv_j)}{v_{0k}-tv_k}
	\end{equation*}
	with $0\le t< t_v$. Since $i,j\in J$ and $k\in J^c$, we know that $(v_{0i}-tv_i)(v_{0j}-tv_j)$ stays bounded away from $0$ and $v_{0k}-tv_k\to 0$ as $t\to t_v$. This implies that $S(\gamma_v(t))\to -\infty$ as $t\to t_v$.
\end{proof}

We also need to control how fast the scalar curvature goes to $-\infty$. For this purpose we prove the following result.

\begin{prop}\label{growthinf}
Consider a stratum $\Delta_J$ such that $\fh\oplus\fm_J$ is not a subalgebra. Then for every $v\in\Sph_J$ and $a>0$, there exist an open neighborhood $U(v)$ in $\Sph$ and a positive number $\epsilon(v)$ such that $S(\gamma_u(t))<-a$ whenever $u\in U(v)$ and $(1-\e(v))t_u<t<t_u$.
\end{prop}

\begin{proof}
Let $A=\frac{\max_ib_i}{\min_i T_i}$ as before. There exist $i,j\in J$ and $k\in J^c$ such that $[ijk]\ne0$. Moreover, $(v_{0i}-t_vv_{i})(v_{0j}-t_vv_{j})>0$ and $(v_{0k}-t_vv_{k})=0$. Define
$$\e(v)=\min\Big\{[ijk]\frac{(v_{0i}-t_vv_i^+)(v_{0j}-t_vv_j^+)}{4(A+a)
t_vv_k},\frac12\Big\},$$
where $v_i^+=\max\{v_i,0\}$.
Evidently, this quantity is always positive. Choose a neighborhood $U(v)$ of $v$ in $\Sph$ such that $u_k>0$ and
\begin{align*}
(v_{0i}&-t_uu_i^+)(v_{0j}-t_uu_j^+)>\tfrac12(v_{0i}-t_vv_i^+)(v_{0j}-t_vv_j^+)\qquad \mbox{and}
\\
v_{0k}&-(1-\e(v))t_uu_k<2(v_{0k}-(1-\e(v))t_vv_k)=2(v_{0k}-t_v v_k +\e(v)t_vv_k)=2\e(v)t_vv_k
\end{align*}
for all $u=(u_1,\ldots,u_r)\in U(v)$. This implies
\begin{align*}
S(\gamma_u(t))&\le A -[ijk]\frac{(v_{0i}-tu_i)(v_{0j}-tu_j)}{v_{0k}-tu_k}
\\
&<A-[ijk]\frac{(v_{0i}-t_uu_i^+)(v_{0j}-t_uu_j^+)}{v_{0k}-(1-\e(v))t_uu_k}\le-a,
\end{align*}
provided $u\in U(v)$ and $(1-\e(v))t_u<t<t_u$.
\end{proof}

If the space $G/H$ has pairwise inequivalent isotropy summands, Proposition~\ref{growthinf} implies the following result, originally proved in~\cite{AP16}. 

\begin{cor}\label{max_H}
	If $\fh$ is maximal in $\fg$,  then $S_{|\M_T}$ attains its global maximum at a metric ${g\in \M_T}$, and hence $\Ric(g)=cT$ for some $c>0$.
\end{cor}

We can add a marking to the strata in $\partial\Delta$. If $\fh\oplus\fm_J=\fk$ is a subalgebra, we denote the stratum $\Delta_J$ by $(\Delta_J, \fk)$ or simply $\Delta_\fk$; if it is not, we denote the stratum by~$(\Delta_J,\infty)$ or~$\Delta_\infty$.

Next, we need an estimate for $S$ near $(\Delta_J, \fk)$. Let $K$ be the connected subgroup of $G$ with Lie algebra~$\fk$. Define
\begin{align*}
\alpha_\fk^D=\sup\{S(h)\mid  h\in \M^D(K/H) \text{ with } \tr_hT_{|\fk\cap\fm}=1\},
\end{align*}
where the letter $D$ is preserved for the decomposition of the tangent space to $K/H$ induced by~$D$. 
Since a normal homogeneous metric has non-negative scalar curvature, $\alpha_\fk^D\ge 0$. Also, $\alpha_\fk^D=0$ if and only if $K/H$ is flat. It is important for us to note that, since $\Delta_\fk$ is not closed in general, the supremum in the definition of $\alpha_\fk^D$ may not be achieved in~$\Delta_\fk$, which will complicate our discussion.

Consider a metric in $\M_T^D$ identified with $y=(y_1,\ldots,y_r)\in\Delta$. If $\fk=\fh\oplus\fm_J$ for some $J\subset I$, then 
\begin{align*}
y=y_{|\fk\cap\fm}+y_{|\fk^\perp},\qquad y_{|\fk\cap\fm}=\sum_{i\in J} \tfrac1{y_i}Q_{|\fm_i},\qquad y_{|\fk^\perp}=\sum_{i\in J^c}\tfrac1{y_i}Q_{|\fm_i}.
\end{align*}
We may regard $y_{|\fk\cap\fm}$ as a metric on $K/H$. Its scalar curvature is given by
\begin{align*}
S(y_{|\fk\cap\fm})=\frac12\sum_{i\in J}d_i\bar b_iy_i-\frac14\sum_{i,j,k\in J}[ijk]\frac{y_iy_j}{y_k},
\end{align*}
where the Killing form of $K$ restricted to $\fm_i$ equals $- \bar b_i Q_{|\fm_i} $. One easily shows that
\begin{equation*}
\bar b_i=b_i-\sum_{j,k\in J^c}\frac{[ijk]}{d_i}.
\end{equation*}

\begin{prop}\label{growthalg}
Consider a stratum $(\Delta_J,\fk)$. If  $y\in\M_T^D$ satisfies $\displaystyle\max_{i\in J^c} y_i\le \e$, then
$$S(y)\le \alpha_\fk^D+\e\sum_{i\in J^c}\frac{d_ib_i}2.$$ 
\end{prop}
\begin{proof}
We break up the formula for the scalar curvature in~\eqref{scalar_y} as follows, using the assumption that $[ijk]=0$ for $i,j\in J$ and $k\in J^c$:
\begin{align*}
S(y)&=\frac12\sum_{i\in J}d_ib_iy_i + \frac12\sum_{i\in J^c}d_ib_iy_i-\frac14\sum_{i,j,k\in J} [ijk]\frac{y_iy_j}{y_k}
\\
&\hphantom{=}~-\frac12\sum_{i\in J}\sum_{j,k\in J^c}[ijk]\frac{y_iy_j}{y_k}
-\frac14\sum_{i\in J}\sum_{j,k\in J^c}[ijk]\frac{y_jy_k}{y_i}-	 \frac14\sum_{i,j, k\in J^c}[ijk]\frac{y_iy_j}{y_k}  \\
&\le\frac12\sum_{i\in J}d_i{\bar{b}}_iy_i+ \frac12 \sum_{i\in J}\sum_{j,k\in J^c}[ijk]y_i  + \frac12\sum_{i\in J^c}d_ib_iy_i-\frac14\sum_{i,j,k\in J} [ijk]\frac{y_iy_j}{y_k} \\
&\hphantom{=}~-\frac14\sum_{i\in J}\sum_{j,k\in J^c}[ijk]y_i\Big(\frac{y_j}{y_k}+\frac{y_k}{y_j}\Big)
\\ &\le S(y_{|\fk\cap\fm}) + \frac12\sum_{i\in J^c}d_ib_iy_i,
\end{align*}
where in the last step we used  the estimate $\frac{y_j}{y_k}+\frac{y_k}{y_j}\ge 2$. Now observe that
\begin{align*}
S(y_{|\fk\cap\fm})=(\tr_{y_{|\fk\cap\fm}}T_{|\fk\cap\fm})S((\tr_{y_{|\fk\cap\fm}}T_{|\fk\cap\fm})y_{|\fk\cap\fm})<S((\tr_{y_{|\fk\cap\fm}}T_{|\fk\cap\fm})y_{|\fk\cap\fm})\le \alpha_\fk^D
\end{align*}
since $\tr_{y_{|\fk\cap\fm}}T_{|\fk\cap\fm}<\tr_yT=1$ and the trace of $T_{|\fk\cap\fm}$ with respect to $(\tr_{y_{|\fk\cap\fm}}T_{|\fk\cap\fm})y_{|\fk\cap\fm}$ equals~1. Consequently,
\begin{align*}
S(y) < \alpha_\fk^D + \frac12\sum_{i\in J^c}d_ib_iy_i.
\end{align*}
When $y_i<\e$ for all $i\in J^c$, we get the desired result.
\end{proof}

We can reformulate \pref{growthalg} as follows.

\begin{cor}\label{local_control} 
Let  $\Delta_\fk^D$ be  a  subalgebra stratum. Then for every $a>\alpha_\fk^D$ there exists a constant $\e>0$ such that the set
$\{g\in\M_T^D \mid S(g)\ge a\}$ does not intersect $  T_\e(\Delta_\fk)$.
\end{cor}

Combining Propositions~\ref{growthinf} and~\ref{growthalg}, we arrive at the following conclusion. 

\begin{cor}\label{cor_est_fixed_D}
Suppose $a>\alpha_\fk^D$ for every subalgebra stratum $\Delta_\fk$. Then $\{g\in\M_T^D \mid S(g)\ge a\}$ is a (possibly empty) compact subset of~$\M_T^D$.
\end{cor}

\pref{growthalg} shows that $\alpha_\fk^D$ is an upper bound for the possible values of the scalar curvature as we approach points in~$\Delta_\fk$. However, it is important to keep in mind that $S$ does not, in general, extend continuously to the closure of~$\Delta$.

We end this section with the following observation. Recall that even if $G$ and $H$ are connected and compact, and if $\fk$ is an intermediate subalgebra, then the connected (intermediate) subgroup with Lie algebra $\fk$ is not necessarily compact. 

\begin{lem}\label{closure}
If $K$ is an intermediate subgroup with Lie algebra~$\fk$, then the closure $\bar K$ is compact and $\alpha_{\bar \fk}^D=\alpha_\fk^D$, where $\bar\fk$ is the Lie algebra of~$\bar K$.
\end{lem}

\begin{proof}
	Let $\fk=\fk_{\mathrm s}\oplus\fz$ with $\fk_{\mathrm s}$ semisimple and $\fz$ the center of $\fk$. Suppose that $K_{\mathrm s}$ and $Z$ are the connected Lie subgroups of $G$ with Lie algebras $\fk_{\mathrm s}$ and~$\fz$. Then $K=K_{\mathrm s}\cdot Z$ (i.e., $K$ is the quotient of $K_{\mathrm s}\times Z$ by a discrete subgroup of the center of~$K_{\mathrm s}\times Z$) and $K_{\mathrm s}$ is compact. The closure $\bar Z$ is compact abelian, and hence a torus. Denote its Lie algebra by~$\bar\fz$. Thus $\bar K=K_{\mathrm s}\cdot \bar Z$. If $g\in\M(K/H)$, then any extension $\bar g\in\M ({\bar K}/H) $ satisfies $S(\bar g)=S(g)$ since $[\bar \fz, \bar \fz]=[\bar \fz,\fk]=0$.
\end{proof}

Thus it is sufficient to compute the invariants $\alpha_\fk^D$ only for intermediate subalgebras for which $K$ is compact.

\bigskip

\section{Riemannian submersions}\label{sec_submersions}

\smallskip

In this section we study the behavior of $S$ near the subalgebra stratum $\Delta_\fk$ geometrically. It will be more convenient to choose the path $g_t$ below instead of $\gamma_v(t)$ since  we can then  use formula~\eqref{scalar_sub} for Riemannian submersions. The goal is to see if there are metrics near $\Delta_\fk$ whose scalar curvature is larger than~$\alpha_\fk^D$. As before, suppose $K$ is an intermediate connected subgroup with Lie algebra~$\fk$ and associated stratum~$\Delta_\fk$. Thus $\mathfrak k=\fh\oplus\fm_J$ for some $J\subset I$. Define
\begin{equation*}
\beta_\fk^D=\sup\{S(h)\mid h\in\M^D(G/K) \text{ with } \tr_hT_{|\fk^\perp}=1\}.
\end{equation*}
We will show that $\beta_\fk-\alpha_\fk$ controls the desired behavior.

We have the homogeneous fibration 
\begin{equation}\label{fibs}
F=K/H\to  G/H \to G/K=B.
\end{equation}
Let us consider metrics on $G/H$ for which this fibration is a Riemannian submersion.
We start with a metric of the form
\begin{equation*}
g=g_F+g_B=\sum_{i\in J}\tfrac1{y_i}Q_{|\fm_i}+\sum_{i\in J^c}\tfrac1{y_i}Q_{|\fm_i}.
\end{equation*}
Assume that $g$ lies in $\M_T$, i.e., $T_1^*+T_2^*=1$, where
\begin{align*}
T_1^*=\tr_{g_F}T_{|\fk\cap\fm}=\sum_{i\in J} d_i y_i T_i,\qquad T_2^*=\tr_{g_B}T_{|\fk^\perp}=\sum_{i\in J^c} d_iy_i  T_i.
\end{align*}
We also require the metric $g_B$ to be $\Ad_K$-invariant so that the projection in~\eqref{fibs} is a Riemannian submersion with $g_F$ and $g_B$ the metrics on the fiber and the base.

Consider the two-parameter family
\begin{equation*}
g_{s,t}=\frac1s g_F +\frac1t  g_B \qquad \text{with}\qquad  sT_1^*+tT_2^*=1.
\end{equation*}
Substituting  $s=\frac{1-tT_2^*}{T_1^*}$, we obtain a one-parameter family of metrics
\begin{align}\label{def_gz}
g_t=\frac{T_1^*}{1-tT_2^*}\,g_F+\frac1t \,g_B
\end{align}
lying in $\M_T^D$.
We call this the {\it canonical variation} associated to $K$.
By~\eqref{scalar_sub}, the scalar curvature of $g_t$ is 
\begin{equation}\label{scal_g}
S(g_t)=s\,S_F +t\, S_B - \frac{t^2}{s}|A|_g =\frac{S_F}{T_1^*}+T_2^*\Big(\frac{S_B}{T_2^*}-\frac{S_F}{T_1^*}\Big)t-\frac{t^2T_1^*}{1-tT_2^*}|A|_g.
\end{equation}
Since $\lim_{t\to0}g_t= \frac{g_F}{T_1^*}\in \Delta_\fk$, every point in $ \Delta_\fk$ is a limit of such a path $g_t$.
Thus we have
\begin{equation}\label{derivative}
\lim_{t\to0}S(g_t)=\frac{S_F}{T_1^*} \qquad\text{and}\qquad
\lim_{t\to0}\frac{dS(g_t)}{d t}=T_2^*\bigg(\frac{S_B}{T_2^*}-\frac{S_F}{T_1^*}\bigg). 
\end{equation}
Notice that 
\begin{equation*}
\frac{S_F}{T_1^*}=S(T_1^*g_F)\le\alpha_\fk^D\qquad\text{and}\qquad \frac{S_B}{T_2^*}=S(T_2^*g_B)\le\beta_\fk^D.
\end{equation*}
We now use these formulas to understand the relationship between the numbers $\alpha_\fk^D$ corresponding to different strata.

\begin{prop}\label{prop_alpha_less_alpha}
	If $\Delta_{\fk'}$ is a stratum adjacent to~$\Delta_\fk$ with $\fk'\subset\fk$, then $\alpha_{\fk'}^D\le\alpha_\fk^D$.
\end{prop}

\begin{proof}
	Let $K'$ be the subgroup of $G$ with Lie algebra~$\fk'$. Thus $H\subset K'\subset K\subset G$. Given $h\in\M^D(K'/H)$ with $\tr_hT_{|\fk'\cap\fm}=1$, define a one-parameter family of metrics
	\begin{align*}
	h_t=\frac{1}{1-t\sum d_iT_i}\,h+\frac1t\,Q_{\fk'^\perp\cap\fk}\in \M^D(K/H),
	\end{align*}
	where the sum is taken over all $i$ with $\fm_i\subset\fk'^\perp\cap\fk$. Applying~\eqref{derivative} to the homogeneous fibration
	$$K'/H\to K/H\to K/K',$$
	we conclude that
	\begin{align*}
	\lim_{t\to 0}S(h_t)=S(h).
	\end{align*}
	This means that, for every $h\in\M^D(K'/H)$ with $\tr_hT_{|\fk'\cap\fm}=1$, there exists a metric $g\in\M^D(K/H)$ with $\tr_gT_{|\fk\cap\fm}=1$ and scalar curvature arbitrarily close to~$S(h)$.
\end{proof}

As we noted in Section~\ref{sec_submersions}, it is possible that the supremum in the definition of $\alpha_\fk^D$ is not attained by a metric in~$\Delta_\fk$. In this case, we have the following result.

\begin{prop}\label{prop_sup_achieved}
Assume that $\alpha_{\fk}^D$ is not attained. Then there exists an adjacent stratum $\Delta_{\fk'}$ such that $\alpha_{\fk'}^D=\alpha_\fk^D$ and $S(h)=\alpha_{\fk'}^D$ for some $h\in\M^D(K'/H)$ with $\tr_hT_{|\fk'\cap\fm}=1$. 
\end{prop}

\begin{proof}
Since $\alpha_{\fk}^D$ is not attained, it is possible to find a sequence $h_i\in \M(K/H)$ with 
$
\lim_{i\to\infty}S(h_i)=\alpha_\fk^D
$
converging to some $h\in\Delta_{J'}$, where the stratum $\Delta_{J'}$ is adjacent to~$(\Delta_J,\fk)$. Applying Proposition~\ref{growthinf} to the homogeneous space $K/H$ and using the nonnegativity of~$\alpha_\fk^D$, we conclude that $\Delta_{J'}=\Delta_{\fk'}$ for some $\fk'\subset\fk$. Similarly, applying Proposition~\ref{growthalg} to $K/H$ shows that
	\begin{align*}
	\alpha_\fk^D=\lim_{i\to\infty}S(h_i)\le\alpha_{\fk'}^D.
	\end{align*}
	In light of Proposition~\ref{prop_alpha_less_alpha}, this means $\alpha_{\fk'}^D=\alpha_{\fk}^D$. If the supremum $\alpha_{\fk'}^D$ is attained, then we are done. Otherwise, we repeat the argument until we reach a subalgebra $\fk''$ for which $\alpha_{\fk''}^D$ is achieved. By \cref{max_H}, this will be the case at the latest for a subalgebra $\fk''$ in which $\fh$ is maximal. 
\end{proof} 

Finally, we show how the difference $\beta_\fk^D-\alpha_\fk^D$ controls the behavior of the scalar curvature functional.

\begin{prop}\label{above}
	Consider a subalgebra stratum $\Delta_\fk$ such that $\alpha_{\fk}^D$ is attained. If $\beta_\fk^D-\alpha_\fk^D>0$, then there exists a metric $g\in\Delta$, arbitrarily close to $\Delta_\fk$, with $S(g)>\alpha_\fk^D$.
\end{prop}

\begin{proof}
Choose $\bar g_F\in\M^D(K/H)$ such that $\tr_{\bar g_F}T_{|\fk\cap\fm}=1$ and $S(\bar g_F)=\alpha_\fk^D$. If $\beta_\fk^D-\alpha_\fk^D>0$, it is  possible to find $\bar g_B\in\M^D(G/K)$ with $\tr_{\bar g_B}T_{|\fk^\perp}=1$ and $S(\bar g_B)-\alpha_\fk^D>0$. Consider the metric $g\in\M_T^D(G/H)$ given by
	\begin{align*}
	g=2(\bar g_F+\bar g_B).
	\end{align*}
	If we let $g_t$ be the canonical variation as in~\eqref{def_gz}, we conclude from~\eqref{derivative} that 
	\begin{align*}
	\lim_{t\to0}S(g_t)=S(\bar g_F)=\alpha_\fk^D\qquad \text{and} \qquad
	\lim_{t\to0}\frac{dS(g_t)}{dt}=\tfrac12(S(\bar g_B)-\alpha_\fk^D)>0.
	\end{align*}
	Clearly, $S(g_t)>\alpha_\fk$ for small~$t$.
\end{proof} 

Combining Propositions~\ref{prop_sup_achieved} and~\ref{above} implies our main theorem if there exists only one decomposition $D$ (up to order of summands).

\begin{prop}\label{global_max_1}
Assume that $G/H$ is a compact homogeneous space such that the modules $\fm_i$ are inequivalent. Let $\fk$ be an intermediate subalgebra of the lowest possible dimension such that $\alpha_{\fk}^D=\sup_{\ell}\alpha_\ell^D$, where the supremum is taken over all intermediate subalgebras~$\ell$. If $\beta_\fk^D-\alpha_{\fk}^D>0$, then $S_{|\M_T}$ achieves its maximum at some metric~$g\in\M_T$, and hence $\Ric(g)=cT$ for some $c>0$.
\end{prop}

It is natural to add an additional marking to the strata of $\partial\Delta$ by labeling a subalgebra stratum $(\Delta_\fk,\alpha_\fk,\beta_\fk)$, which encodes the behavior of $S$ in a neighborhood of $\Delta_\fk$.

As remarked before, if $\fk$ is an intermediate subalgebra, then the connected subgroup $K$ with Lie algebra $\fk$ may not be compact. Thus, $G/K$ is not necessarily a manifold, and hence \eqref{fibs} is not an actual Riemannian submersion. Nevertheless, \eqref{scal_g} still holds since, by homogeneity, this is a local formula and \eqref{fibs} is still a Riemannian submersion locally. Thus, $\beta_{\fk}^D$ is well-defined for any intermediate subalgebra.

\bigskip

\section{Global maxima}\label{sec_Maxima}

\smallskip

From now on, we allow the decomposition $D$ to vary. Clearly, the numbers $b_i$ and the structure constants $[ijk]$ depend continuously on~$D$. Recall also that the space $\mathcal D$ of all decompositions is compact.

Consider an intermediate subgroup $K$ with Lie algebra~$\fk$. The numbers $\alpha_\fk^D$ and $\beta_\fk^D$ introduced above depend on the choice of the decomposition~$D$. Removing this dependence, we define
\begin{align*}
\alpha_\fk&=\sup \{S(g)\mid g\in \M(K/H)\text{ with }\tr_g(T_{|\fk\cap\fm})=1\}\qquad\mbox{and} \\
\beta_\fk&=\sup \{S(g)\mid g\in \M(G/K)\text{ with }\tr_g(T_{|\fk^\perp})=1\},
\end{align*}
and introduce an invariant for $G/H$ given by
\begin{equation*}\label{alpha_GH}
\alpha_{G/H}=\textstyle\sup_{\fk}\alpha_\fk,
\end{equation*}
where the supremum is taken over all intermediate subalgebras~$\fk$.

Our first goal is to extend Propositions~\ref{growthinf} and~\ref{growthalg} to all of~$\M_T$. We will  use the following parametrization of the space~$\M$, convenient in our context. Specifically, consider the map ${\sigma:\mathbb R_+^r\times\DD\to\M}$ defined by
$$
\sigma(y,D)=\sum_i\tfrac1{y_i}Q_{|\mathfrak m_i},
$$
where $y=(y_1,\ldots,y_r)$ and $D$ is the decomposition with irreducible modules $\mathfrak m_1,\ldots,\mathfrak m_r$. This map is clearly continuous. While it is surjective, it may not be injective. The preimages of some metrics are infinite when some of the isotropy summands of $G/H$ are equivalent. However, given a metric $g\in\M$, the preimage $\sigma^{-1}(g)$ is compact.

To state our next result, we fix a point $(y,D)$ in the boundary of~$\sigma^{-1}(\M_T)$. Let $\Delta$ be the simplex associated with the decomposition~$D$. Clearly, $y$ lies in a stratum $\Delta_J$ for some $J\subset I$. The following result generalises Proposition~\ref{growthalg} to all of~$\M_T$.

\begin{prop}\label{local_control_3}
Assume that $y$ lies in a subalgebra stratum $(\Delta_J,\fk)$. Then for every $\e>0$ there exists an open neighborhood $U$ of $(y,D)$ in $\R^r\times\DD$ such that
\begin{align*}
S(\sigma(y',D'))\le\alpha_\fk+\epsilon
\end{align*}
whenever $(y',D')\in U\cap \sigma^{-1}(\M_T)$.
\end{prop}

\begin{proof}
We use the notation $b_i$, $T_i$ and $[ijk]$ (respectively, $b_i'$, $T_i'$ and $[ijk]'$) for the constants associated with the decomposition~$D$ (respectively,~$D'$). Given $\delta>0$, there exists a neighborhood $U_\delta$ of $(y,D)$ in $\R^r\times\DD$ such that
\begin{align*}
|y_i-y_i'|+|b_i-b_i'|+|T_i-T_i'|+|[ijk]-[ijk]'|<\delta
\end{align*}
for all~$i,j,k$ whenever $(y',D')\in U_\delta\cap\sigma^{-1}(\M_T)$. Let us choose $\delta$ small enough to ensure that $y_i'>\frac{y_i}{2}$ in this set. The trace constraint implies that  $y_i'<\frac{1}{d_iT_i'}$. Notice also that there exist common lower and upper bounds for $T_i$ independent of~$D$.

If $(y',D')\in U_\delta\cap \sigma^{-1}(\M_T)$, we find, as in the proof of \pref{growthalg}, that
\begin{align*}
S(\sigma(y',D'))&=\frac12\sum_{i}d_ib_i'y_i'-\frac14\sum_{i,j,k}\frac{y_i'y_j'}{y_k'}[ijk]'
\\
&\le
\frac12\sum_{i\in J}\bigg(d_ib_i'-\sum_{j,k\in J^c}[ijk]'\bigg)y_i'-\frac14\sum_{i,j,k\in J}\frac{y_i'y_j'}{y_k'}[ijk]'+\frac12\sum_{i\in J^c}d_ib_i'y_i'
\\
&\le
\frac12\sum_{i\in J}d_i\bar b_iy_i'-\frac14\sum_{i,j,k\in J}\frac{y_i'y_j'}{y_k'}[ijk]+\frac12\sum_{i\in J}d_i(b_i'-b_i)y_i'
\\
&\hphantom{=}~-\frac12\sum_{i\in J}\sum_{j,k\in J^c}([ijk]'-[ijk])y_i'-\frac14\sum_{i,j,k\in J}\frac{y_i'y_j'}{y_k'}([ijk]'-[ijk])+\frac12\sum_{i\in J^c}d_ib_i'y_i'.
\end{align*}
Consequently, for small enough $\delta$, since $y_i=0$ when $i\in J^c$, we have
\begin{align*}
S(\sigma(y',D'))&\le
S(\sigma(y',D)_{|\fk\cap\fm})+\frac12\sum_{i\in J}\frac{|b_i'-b_i|}{T_i'}
\\
&\hphantom{=}~+\frac12\sum_{i\in J}\sum_{j,k\in J^c}\frac{|[ijk]'-[ijk]|}{d_iT_i'}+\frac12\sum_{i,j,k\in J}\frac{|[ijk]'-[ijk]|}{d_id_jT_i'T_j'y_k}+\frac12\sum_{i\in J^c}d_ib_i'|y_i'-y_i|
\\
&\le S(\sigma(y',D)_{|\fk\cap\fm})+\frac\epsilon2.
\end{align*}
Shrinking $\delta$ further if necessary and using the continuity of the scalar curvature, we conclude that 
$$
S(\sigma(y',D)_{|\fk\cap\fm})<S(\sigma(y,D))+\frac{\e}2\le\alpha_\fk+\frac{\e}2,
$$
which implies the result.
\end{proof}

Using a similar (but simpler) proof, we can generalize Proposition~\ref{growthinf}.

\begin{prop}\label{alg_control_3}
Assume that $(y,D)$ lies in a stratum $(\Delta_J,\infty)$. Given $a>0$, there exists a neighborhood $U$ of $(y,D)$ in $\R^r\times\DD$ such that $S(\sigma(y',D'))<-a$ whenever $(y',D')\in U\cap \sigma^{-1}(\M_T)$.
\end{prop}

The precompactness of $\sigma^{-1}(\M_T)\subset\R^r\times\DD$ and Propositions~\ref{local_control_3} and~\ref{alg_control_3} yield the following extension of Corollary~\ref{cor_est_fixed_D} to all of~$\M_T$.

\begin{cor}\label{cor_equiv_est}
If $a>\alpha_{G/H}$, then $\{g\in\M_T \mid S(g)\ge a\}$ is a compact subset of~$\M_T$.
\end{cor}

Our next result generalises Proposition~\ref{prop_sup_achieved}.

\begin{prop}\label{global_achieved}
Assume that $\fh$ is not maximal in $\fg$. Then there exists an intermediate subgroup $K$ with Lie algebra $\fk$  such that  $\alpha_{\fk}  =\alpha_{G/H}$. If $K$ has the least possible dimension of all such subgroups, then there exists $h\in\M(K/H)$ with
\begin{align*}
S(h)=\alpha_{\fk}\qquad\text{and}\qquad \tr_hT_{|\fk\cap\fm}=1.
\end{align*}
\end{prop}

\begin{proof}
For each $i\in\N$, choose a subgroup $K_i\subset G$ with Lie algebra $\fk_i$ such that
$\alpha_{\fk_i}\ge\alpha_{G/H}-\frac1{2i}$. Assume that $K_i$ has  the least possible dimension among all  subgroups with this property.
Let $h_i$ be a metric on $K_i/H$ with
$$
S(h_i)>\alpha_{G/H}-\tfrac1i \qquad\text{and}\qquad \tr_{h_i}T_{|\fk_i\cap\fm}=1.
$$ 
We may assume that $\fk_i$ converge to an intermediate subalgebra $\fk$ and that $\dim\fk_i=\dim\fk$ for all $i$.  Our goal is to show that a subsequence of $(h_i)$ converges to a metric $h\in\M(K/H)$.

Consider a decomposition $D_i\in\DD$ with modules $\mathfrak m_1^i,\ldots,\mathfrak m_r^i$ such that 
$$
\fk_i=\fh\oplus \bigoplus_{j\in J_i}\mathfrak m_j^i\qquad\text{and}\qquad h_i(\mathfrak m_k^i,\mathfrak m_l^i)=0
$$
for some $J_i\subset I$ and all $k,l\in J_i$ with $k\ne l$. Passing to a subsequence if necessary, we may assume that these decompositions converge to some $D\in\DD$ with modules $\mathfrak m_1,\ldots,\mathfrak m_r$ and that $J_i$ does not depend on~$i$ (thus we can omit the index~$i$ from the notation~$J_i$). Let $K$ be the connected subgroup of $G$ with Lie algebra $\fk=\fh\oplus\mathfrak m_J$.  There exist positive numbers~$x_{ji}$ such that
\begin{align*}
h_i=\sum_{j\in J}x_{ji}Q_{|\mathfrak m_j^i}.
\end{align*}
Note that $\alpha_{G/H}>\sup_{\ell}\alpha_\ell$, where the supremum is taken over all intermediate subalgebras $\ell$ between $\fh$ and~$\fk$. Indeed, if not, there exists a subalgebra $\ell$ with $\alpha_\ell > \alpha_{G/H} - \frac{1}{2i}$, contradicting the assumption that $K_i$ is chosen to be of smallest possible dimension. Therefore,
$$
S(h_i)>\alpha_{G/H} - \tfrac{1}{i}>\sup\nolimits_{\ell}\alpha_\ell
$$
for large~$i$. We now claim that the constants $x_{ji}$ all lie in some compact subset of~$\R_+$. To see this, we can argue as in 
Propositions~\ref{local_control_3} and \ref{alg_control_3} and \cref{cor_equiv_est} replacing $G/H$ with the sequence $(K_i/H)$ and using the fact that the structure constants of $\fk_i$ converge to those of~$\fk$.

Thus, passing to a subsequence  if necessary, we may assume that
$$
\lim_{i\to\infty}x_{ji}=x_i\in\R_+
$$
for $j\in J$. The metric $h=\sum_{i\in J}x_iQ_{|\mathfrak m_i}$ satisfies  $S(h)=\alpha_{\fk}=\alpha_{G/H}$ and $\tr_hT_{|\fk\cap\fm}=1$.
\end{proof}

We are now ready to prove our main theorem.

\begin{proof}[Proof of the main theorem] The existence of the subgroup $K$ follows from \pref{global_achieved}.
By~\pref{above}, there exists $\e>0$ such that the superlevel set $\{g\in\M_T\mid S(g)\ge \alpha_{\fk}+\e\}$ is nonempty. Corollary~\ref{cor_equiv_est} implies that this set is also compact. Consequently, $S_{|\M_T}$ assumes its maximum at a metric~$g\in\M_T$. As a critical point, such a metric satisfies $\Ric(g)=cT$ for some constant $c>0$.
\end{proof}

\bigskip

\section{Examples}\label{sec_Example}

\smallskip

\subsection{The Stiefel manifold $\mathbf{\emph{V}_2(\R^4)}$.}

This example is interesting since some of the modules $\fm_i$ are equivalent and hence we need to consider non-diagonal metrics as well. Furthermore, the set of decompositions is not discrete; in fact, as we will see below, $\DD$ is two-dimensional. In addition, there exists a circle of intermediate subgroups and hence we need to maximize~$\alpha_\fk$. This is  the first example of this type where the existence of global maxima was studied.

Consider the homogeneous space $G/H=(SU(2)\times SU(2))/S^1$ where the circle group is embedded diagonally into $SU(2)\times SU(2)$. The bi-invariant metric $Q$ we choose on $G$ is such that 
$
|(X,Y)|_Q=-\tfrac12(\tr(X^2)+\tr(Y^2))
$. Since the two-fold cover $SU(2)\times SU(2)\to SO(4)$ sends the diagonal embedding $\diag(SU(2))\subset SU(2)\times SU(2)$ to $SO(3)\subset SO(4)$, the space $G/H$ coincides with the Stiefel manifold $V_2(\mathbb R^4)=SO(4)/SO(2)$. 

We can identify $SU(2)$ and $\fsu(2)$ with the group of unit quaternions and the Lie algebra of purely imaginary quaternions. Then
$$
H=\{(e^{\eta\bf i},e^{\eta\bf i})\,|\,\eta\in[0,2\pi)\}\qquad\mbox{and}\qquad 
\fh=\spam\{(\,\mathbf{i},\,\mathbf{i})\}.
$$
Consider the following $\Ad_H$-invariant decomposition of $\fm$:
\begin{align}\label{isolated_decom_Stiefel}
	\fm_0=\spam\{(\,\mathbf{i},-\mathbf{i})\},\qquad \fm_1=\{(z\,\bj,0)\mid z\in\C \}\qquad\mbox{and}\qquad\fm_2=\{(0,z\,\bj)\mid z\in\C\}.
\end{align}
Under $\Ad_H$, the element $(e^{\eta\bf i},e^{\eta\bf i})\in H$ takes $(z{\bf j},0)$ and $(0,z\bf j)$ to $(e^{2\eta\bf i}z{\bf j},0)$ and $(0,e^{2\eta\bf i}z\bf j)$, respectively. This implies that the restrictions of $\Ad_H$ to $\fm_1$ and $\fm_2$ are equivalent complex representations. 
Using the $Q$-orthonormal bases $\big\{\frac{({\bf i},-{\bf i})}{\sqrt2}\big\}$ of~$\fm_0$,  $\{({\bf j},0), ({\bf k},0)\}$ of~$\fm_1$, and $\{(0,{\bf j}), (0,{\bf k})\}$ of~$\fm_2$, we can thus represent the metric $g\in\M$ and the tensor field $T$ by the matrices
\begin{equation}\label{Stiefel_metric}
	g=\left(\begin{matrix}
		x_0&0&0&0&0\\
		0&x_1&0&x_3&x_4\\
		0&0&x_1&-x_4&x_3\\
		0&x_3&-x_4&x_2&0\\
		0&x_4&x_3&0&x_2
	\end{matrix}
	\right)\qquad\mbox{and}\qquad
	T=
	\left(
	\begin{matrix}
		T_0&0&0&0&0\\
		0&T_1&0&T_3&T_4\\
		0&0&T_1&-T_4&T_3\\
		0&T_3&-T_4&T_2&0\\
		0&T_4&T_3&0&T_2
	\end{matrix}
	\right),
\end{equation}
which must be positive-definite. This is the case if $x_0,x_1,x_2>0$ and $x_1x_2-x_3^2-x_4^2>0$ and similarly for~$T$. The outer automorphism of $G=SU(2)\times SU(2)$ that switches the two factors preserves $H$ and hence induces an isometry of $(V_2(\R^4),Q)$. It acts on $g$ by taking $(x_1,x_2,x_3,x_4)$ to $(x_2,x_1,x_3,-x_4)$, and similarly on~$T$. Conjugation by $(\bf{j},\bf{j})$ preserves $H$ as well. It acts on $g$ by taking $(x_1,x_2,x_3,x_4)$ to $(x_1,x_2,x_3,-x_4)$. Thus the composition takes $(x_1,x_2,x_3,x_4)$ to $(x_2,x_1,x_3,x_4)$, which shows that we may assume $T_1\ge T_2$. It will be convenient for us to denote
\begin{align*}
\gamma(T_0,T_2,T_3,T_4)=\frac{T_0+\sqrt{T_0^2+16T_0T_2}+16\sqrt{T_3^2+T_4^2}}8.
\end{align*}
As we will see shortly, the maximal subgroups of $G$ containing $H$ are the subgroups $K_1= SU(2)\times S^1$ and $K_2= S^1\times SU(2)$ and a one-parameter family of three-dimensional subgroups $K^\theta$, $\theta\in[0,2\pi)$, isomorphic to $SU(2)$. Our main theorem leads to the following result. We distinguish the case where the supremum $\alpha_{G/H}$ is attained by one of the four-dimensional subgroups from the case where it is attained by a three-dimensional subgroup.

\begin{prop}\label{V2R4}
	Suppose that $G/H=V_2(\R^4)$ and $T_1\ge T_2$. Then we have: 
	\begin{itemize}
		\item[\emph{(a)}]$\alpha_{G/H}=\alpha_{\fk_2}$ and $\beta_{\fk_2}-\alpha_{\fk_2}>0$ if and only if
		$$
		\gamma(T_0,T_2,T_3,T_4)\le T_1<T_2+\frac{T_0+\sqrt{T_0^2+16T_0T_2}}8.
		$$
		\item[\emph{(b)}] 
		$\alpha_{G/H}=\alpha_{\fk^\theta}$ and $\beta_{\fk^\theta}-\alpha_{\fk^\theta}>0$ for some $\theta\in[0,2\pi)$ if and only if 
		$$
		\frac{T_0}2-T_2+4\sqrt{T_3^2+T_4^2}<T_1\le\gamma(T_0,T_2,T_3,T_4).
		$$
	\end{itemize}
	In both cases, $S_{|\M_T}$ attains its global maximum.
\end{prop}

\begin{proof}
	We start by describing the space  of all decompositions. Since the representations of $H$ on $\fm_1$ and $\fm_2$ are complex equivalent representations, $\DD$ is the two-parameter family of decompositions $D_{s,\theta}$ with modules
	\begin{align}\label{decomp_last}
		\fm_0=\spam\{(\,\mathbf{i},-\,\mathbf{i})\}, \quad
		\fm_1^{s,\theta}=\{(se^{\theta\bi}z\bj,-tz\,\bj)\mid z\in\C \}\quad\mbox{and}\quad \fm_2^{s,\theta}=\{(te^{\theta\bi}z\,\bj,sz\,\bj)\mid z\in\C \},
	\end{align}
	where $s^2+t^2=1$.
	Let us choose a $Q$-orthonormal basis in $\fm_1^{s,\theta}$ consisting of the vectors
	\begin{align}\label{basis_m1}
		v_1=(s\cos\theta\, \bj+s\sin\theta\,\bk,-t\,\bj)\qquad\mbox{and}\qquad v_2=(-s\sin\theta\, \bj+s\cos\theta\,\bk,-t\,\bk)
	\end{align}
	and one in $\fm_2^{s,\theta}$ consisting of
	\begin{align}\label{basis_m2}
		w_1=(t\cos\theta\, \bj+t\sin\theta\,\bk,s\bj)\qquad\mbox{and}\qquad w_2=(-t\sin\theta\,\bj+t\cos\theta\,\bk,s\bk).
	\end{align}	
	Using these bases, for the decomposition~$D_{s,\theta}$, one easily computes
	\begin{equation}\label{str_cons_Stiefel}
		[011]=[022]=4(s^2-t^2)^2\qquad \mbox{and} \qquad [012]=8s^2t^2.
	\end{equation}
	The structure constants unrelated to these by permutation are~0.
	
	The decomposition corresponding to $s=1, t=0$ gives rise to the natural set of diagonal metrics $x_3=x_4=0$ with  intermediate subgroups
	$$
	K_0=T^2,\qquad K_1=SU(2)\times S^1_{\textrm{rt}}\qquad\mbox{and}\qquad K_2=S^1_{\textrm{lt}}\times SU(2)
	$$
	with Lie algebras
	\begin{align*}
		\fk_0 
		=\fh\oplus\fm_0, \qquad
		\fk_1=
		\fh\oplus\fm_0\oplus\fm_1 \qquad \mbox{and} \qquad 
		\fk_2
		=\fh\oplus\fm_0\oplus\fm_2.
	\end{align*}
	If $s=t=\frac{\sqrt2}2$, we have, for each~$\theta$, the intermediate subgroup 
	$$
	K^\theta\simeq SU(2) \qquad \text{with Lie algebra}\qquad  \fk^\theta	=\fh\oplus\fm_1^{s,\theta}.
	$$
	The remaining decompositions do not produce any subgroups. Thus, there are three isolated intermediate subgroups, $K_0$, $K_1$ and $K_2$, as well as a one-parameter family of subgroups~$K^\theta$.
	
	The identity component of the normalizer of $H$ is given by  $N_0(H)=T^2\subset SU(2)\times SU(2),$ and hence $N_0(H)/H\simeq S^1$, represented by elements of the form $ (e^{\eta\bf i},1)\in N_0(H)$. These elements act via right translation on $G/H$ and via conjugation on~$\fm$. Thus they also act on $\DD$ and, via pullback, on~$\M$. It is easy to see that $(e^{\eta\bf i},1)$ takes $D_{s,\theta}$ to $D_{s,\theta+2\eta}$.  This implies, in particular, that the subalgebras $\fk^\theta$ are all conjugate to each other by elements of~$N_0(H)$. Since $N_0(H)$ acts on $G/H$ by isometries in~$Q$, it follows that $(G/K^\theta,Q)$, as well as $(K^\theta/H,Q)$, are all isometric to each other.
	
	We now compute the constants $\alpha_\fk$ and $\beta_\fk$ for the intermediate subalgebras. For the maximal subgroups $K_1$ and $K_2$ we have
	\begin{align*}
		\alpha_{\fk_i}=\frac{8T_i+ T_0-\sqrt{ T_0^2+16  T_0 T_i} }{2T_i^2}\qquad\mbox{and}\qquad \beta_{\fk_i}=\frac{4}{T_j},
	\end{align*}
	where $(i,j)$ is a permutation of $\{1,2\}$. Indeed, $G/K_i$ are isotropy irreducible, which easily determines~$\beta_{\fk_i}$. The space $K_i/H$ has two irreducible summands, and the scalar curvature, under the trace constraint, is $$S=8y_1-\frac{y_1^2}{y_0}.$$
The above value $\alpha_{\fk_i}$ is its maximum. The assumption $T_1\ge T_2$ implies that $\alpha_{\fk_2}\ge\alpha_{\fk_1}$. Furthermore, $\beta_{\fk_2}-\alpha_{\fk_2}>0$ if and only if
	\begin{eqnarray}\label{Stiefel_max1}
		T_1<T_2+\frac{T_0+\sqrt{T_0^2+16T_0T_2}}8.
	\end{eqnarray}
	
	For the three-dimensional subalgebras $\fk^\theta$, we use the bases~\eqref{basis_m1} and~\eqref{basis_m2}. Now, the parameters $s$ and $t$ both equal $\frac1{\sqrt2}$. The tensor field $T$ satisfies
	\begin{align}\label{T_transform}
		2T(v_1,v_1)&=2T(v_2,v_2)=T_1+T_2-2(\cos\theta\, T_3-\sin\theta\, T_4), \notag \\
		2T(w_1,w_1)&=2T(w_2,w_2)=T_1+T_2+2(\cos\theta\, T_3-\sin\theta\, T_4), \notag  \\
		2T( v_1,w_1)&=2T( v_2,w_2)=T_1-T_2, \notag  \\
		T( v_1,w_2)&=-T( v_2,w_1)=\cos\theta\, T_4+\sin\theta\, T_3,
	\end{align}
	and $T( v_1,v_2)=T( w_1,w_2)=0$. The action of the quotient $N_0(H)/H$ establishes an isometry between $(K^\theta/H,Q)$ and the space $(\diag(SU(2))/\diag(S^1),Q)$, which one easily sees has scalar curvature~$4$. The trace constraint means that $g(v_1,v_1)=2T(v_1,v_1)$, implying
	\begin{align*}
		\alpha_{\fk^\theta}=\frac{4}{T_1+T_2-2(\cos\theta\, T_3-\sin\theta\, T_4)}.
	\end{align*}
	
	The action of the normalizer shows that $(G/K^\theta,Q)$ are all isometric to the symmetric space $((SU(2)\times SU(2))/\diag(SU(2)),Q)$. Thus, they have scalar curvature~12. The trace constraint for $K^\theta$-invariant metrics on $G/K^\theta$ takes the form
	\begin{align*}
		g\bigg(\frac{({\bf i},-{\bf i})}{\sqrt{2}},\frac{(\bf i,-\bf i)}{\sqrt{2}}\bigg)=g(w_1,w_1)&=g(w_2,w_2)
		\\
		&=T\bigg(\frac{({\bf i},-{\bf i})}{\sqrt{2}},\frac{(\bf i,-\bf i)}{\sqrt{2}}\bigg)+T(w_1,w_1)+T(w_2,w_2)
		\\
		&=T_0+T_1+T_2+2(\cos\theta\, T_3-\sin\theta\, T_4).
	\end{align*}
	Thus
	\begin{equation*}
		\beta_{\fk^\theta}=\frac{12}{T_0+T_1+T_2+2(\cos\theta\, T_3-\sin\theta\, T_4)}.
	\end{equation*}
	
	We now choose an angle $\theta_0$ such that $\alpha_{\fk^{\theta_0}}$ is maximal, i.e., $\alpha_{\fk^{\theta_0}}=\sup_{\theta}\alpha_{\fk^\theta}$. Observe that
	$$\cos\theta\, T_3-\sin\theta\, T_4=\sqrt{T_3^2+T_4^2}\,\sin(\theta-\eta)$$
	for some phase shift $\eta$. The largest possible value of this quantity is $\sqrt{T_3^2+T_4^2}$, which means
	\begin{equation*}
		\alpha_{\fk^{\theta_0}}=\frac{4}{T_1+T_2-2\sqrt{T_3^2+T_4^2}}\qquad \mbox{and} \qquad \beta_{\fk^{\theta_0}}=\frac{12}{T_0+T_1+T_2+2\sqrt{T_3^2+T_4^2}}.
	\end{equation*}
	Clearly, $\beta_{\fk^{\theta_0}}-\alpha_{\fk^{\theta_0}}>0$ if and only if 
	\begin{equation}\label{Stiefel_max_2}
		T_0<2T_1+2T_2-8\sqrt{T_3^2+T_4^2}.
	\end{equation}
	
	Let us determine $\alpha_{G/H}$. As noted above, the assumption $T_1\ge T_2$ implies that $\alpha_{\fk_2}\ge\alpha_{\fk_1}$. Elementary analysis shows that $\alpha_{\fk_2}\ge\alpha_{\fk^{\theta_0}}$ if and only if $T_1\ge\gamma(T_0,T_2,T_3,T_4)$. Combining this condition with~\eqref{Stiefel_max1}, we arrive at statement~(a) of the proposition. Reversing the inequality and taking note of~\eqref{Stiefel_max_2}, we obtain statement~(b).
\end{proof}

\begin{rem}\label{rem_diag_univ}
	For the set of diagonal metrics $x_3=x_4=0$ one easily solves the equation $\Ric(g)=cT$ directly.  This equation reduces to the system
	\begin{align}\label{sys_direct}
		x_0=(4-cT_1)x_1,
		\qquad
		x_0=(4-cT_2)x_2,
		\qquad 
		cT_0&=(4-cT_1)^2+(4-cT_2)^2.
	\end{align}
One easily shows that the solution is unique.	In fact, our main theorem, applied to such diagonal metrics, shows that the critical point  is always a global maximum. The action of $N_0(H)/H$ fixes the diagonal metrics and acts by isometries. Thus the Ricci curvature of a diagonal metric must be diagonal as well. In particular, if $T_3=T_4=0$, then every solution to the system~\eqref{sys_direct}  is a critical point of $S_{|\M_T}$ on all of $\M_T$. In Figure~\ref{Stiefel_regions} these metrics lie in the union of the 3 grey regions. On the other hand, given a critical point of $S_{|\M_T}$ with $T_3=T_4=0$ and $x_3\ne0$, one obtains a circle of further critical points by applying the normalizer.
\end{rem}

For a metric $g$ as in~\eqref{Stiefel_metric}, the constraint $\tr_gT=1$ becomes
\begin{equation*}
	\frac{T_0}{x_0}+\frac{ 2x_1T_2 +2x_2T_1-4x_3T_3-4x_4T_4}{\Lambda}=1,
\end{equation*}
and the scalar curvature satisfies
\begin{equation*}
	S(g)=\frac{8}{x_0}+\frac{8(x_1+x_2)}{\Lambda}-\frac{8x_1x_2}{x_0\Lambda}
	-\frac{x_0(x_1-x_2)^2}{\Lambda^2} - \frac{2x_0}{\Lambda},
\end{equation*}
where $\Lambda=x_1x_2-x_3^2-x_4^2$; see, e.g.,~\cite{BWZ04}.  We may assume that $T_0=1$. If $T_3^2+T_4^2=0$, we show in Figure~\ref{Stiefel_regions} points in the $(T_1,T_2)$-plane that correspond to various behaviors of~$S_{|\M_T}$. In the union of the three grey regions the solutions to \eqref{sys_direct} guarantee the existence of a (diagonal) critical point. In the dark-grey region $\fk_2$ yields a global maximum, and in the middle-grey region ${\fk_\theta}$ . A computer-assisted experiment shows that this global maximum is always a diagonal metric; see Example~\ref{exa_max+saddle}. However, in the light grey region the diagonal critical point might not be a global maximum, as we demonstrate below.

If $T_3^2+T_4^2\ne0$, the picture is the same but with the line $2T_1=1-2T_2$ and the curve $8T_1=1+\sqrt{1+16T_2}$ shifted to the right and the reflection of $8T_1=1+\sqrt{1+16T_2}$ shifted up. 

\begin{figure}
	\centering
	\includegraphics[width=80mm]{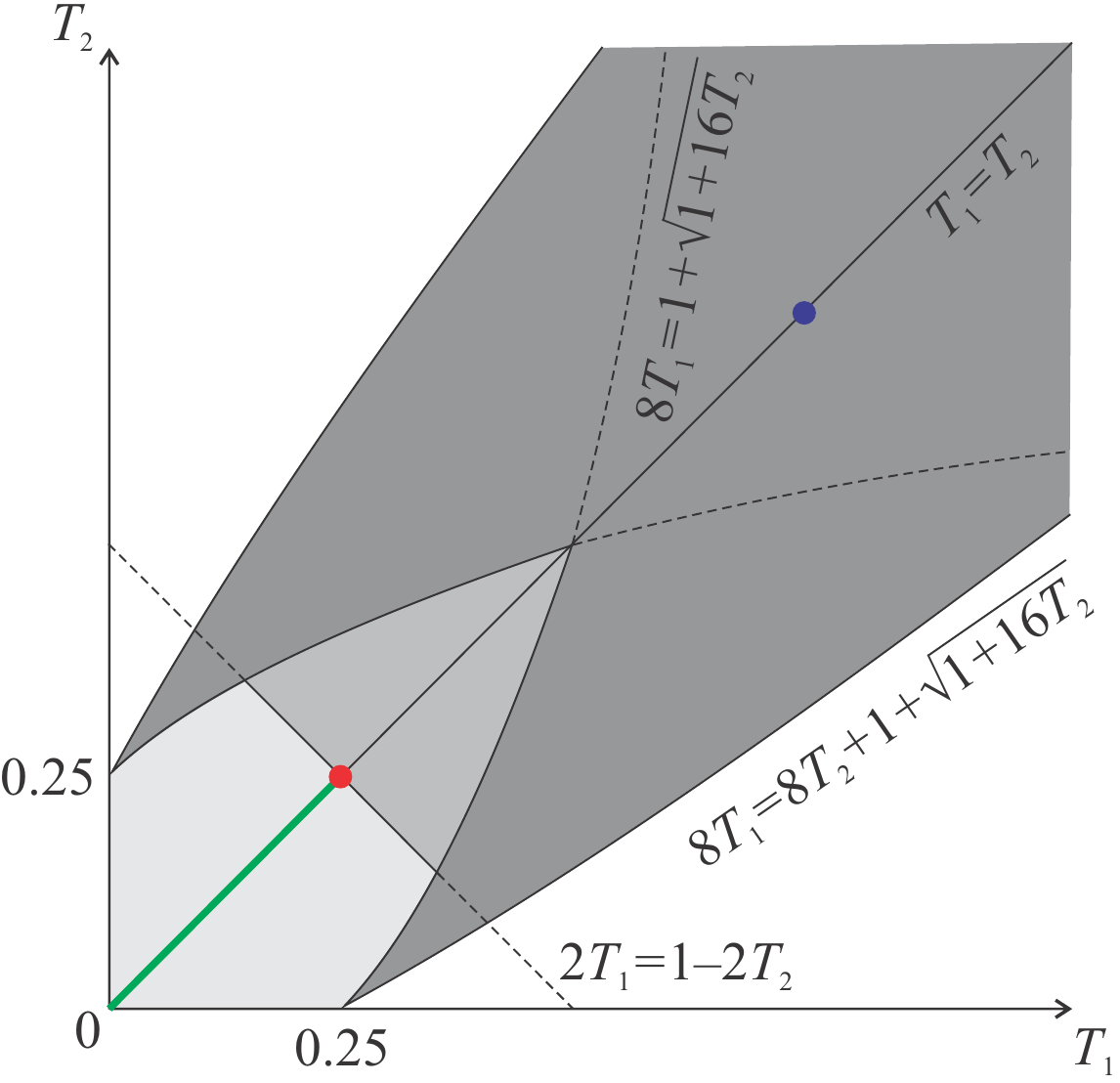}
	\caption{Prescribed Ricci curvature on the Stiefel manifold $V_2(\R^4)$}
	\label{Stiefel_regions}
\end{figure}

\bigskip

Next, we give several examples that illustrate various behaviors.

\begin{example}\label{saddle_stiefel}
	Suppose that $T_0=1$, $T_1=T_2=t>0$ and $T_3^2+T_4^2=0$. A straightforward computation shows that the diagonal metrics with
	\begin{align}\label{Stiefel_diag_crit}
		x_0=\sqrt{32t + 1},\qquad x_1=x_2=\frac{32t+1+\sqrt{32t+1}}{8} \qquad\mbox{and}\qquad  x_3=x_4=0
	\end{align}
	solve the system \eqref{sys_direct} and hence are critical points of the functional $S_{|\M_T}$ when restricted to the set of diagonal metrics.
	By Remark~\ref{rem_diag_univ}, they are also critical points of the scalar curvature functional on all of $\M_T$.
	Maple shows that all four eigenvalues of the Hessian are negative for $t>\frac14$, i.e.,~\eqref{Stiefel_diag_crit} defines a local maximum. By the Maple experiment in Example~\ref{exa_max+saddle}, it is a global maximum. On the other hand, if $t<\frac14$, we have two negative and two positive eigenvalues, which means~\eqref{Stiefel_diag_crit} is a saddle. Thus global maxima among the set of diagonal metrics in $\M_T$ can, in fact, be saddles on all of~$\M_T$. 
	
	We point out that~\eqref{Stiefel_diag_crit} is the Jensen Einstein metric when $t=\frac34$, and a global maximum. It is marked with the blue dot in Figure~\ref{Stiefel_regions}. 
\end{example}

\begin{example}\label{surface_critical}
	From the previous example we see that the metric with $t=\frac14$ must be special. Thus we choose $T_0=1$, $T_1=T_2=\frac14$ and $T_3^2+T_4^2=0$. This choice of $T$ is marked by the red dot in Figure~\ref{Stiefel_regions}. Direct verification shows that the metrics $g\in\M_T$  with
	\begin{equation}\label{disc_crit_pts}
		x_0=2t,\qquad	x_1=x_2=t,\qquad x_3 = t\sqrt{\frac{2t - 3}{2t - 1}}\cos\psi\qquad\mbox{and}\qquad x_4 = t\sqrt{\frac{2t - 3}{2t - 1}}\sin\psi
	\end{equation}
	are critical points of $S_{|\M_T}$ for $t\in\big[\frac32,\infty\big)$ and $\psi\in [0,2\pi)$ with scalar curvature $8$.  They form a surface diffeomorphic to $\R^2$, described in the coordinates $(t,\psi)$. The normalizer $N_0(H)/H$ acts on this surface via $(t,\psi)\to (t,\psi+2\eta)$. Thus metrics with the same value of $\psi$ are isometric. On the other hand, the squared volume of the metric~\eqref{disc_crit_pts} equals $\frac{8t^5}{(2t-1)^2}$, and hence metrics with different values of $t$ are not isometric. To determine the critical point type of~\eqref{disc_crit_pts}, we compute the eigenvalues of the Hessian of~$S_{|\M_T}$. Two of them are always negative, and the other two vanish. The 0-eigenspace is tangent to the surface of critical points. Consequently, this surface is a non-degenerate critical submanifold with index~2. Using the Morse--Bott lemma, we conclude that it is isolated and a local maximum. Using the numerical discussion in Example 3,  we conclude that it must in fact be a  global maximum since otherwise the metrics in Example~\ref{saddle_stiefel} with $t$ near $1/4$ would have to have scalar curvature larger than~$8$.
\end{example}

\begin{example}\label{exa_max+saddle}
Assume that $T_3^2+T_4^2=0$. As explained below,  $S_{|\M_T}$ has a non-diagonal critical point if and only if $T$ lies in the pink region in Figure~\ref{Stiefel_numeric}. Since the action of $N_0(H)/H$ leaves $T$ unchanged, we obtain a circle of non-diagonal critical points for each such~$T$. It is always a non-degenerate critical submanifold of index~2 and co-index~1. In addition, we have a diagonal critical point for $T$ between the thick curves. By Proposition~\ref{V2R4}, this diagonal critical point must be a global maximum when $T$ lies in the dark-grey or the middle-grey region in Figure~\ref{Stiefel_regions}. Computation of the eigenvalues of the Hessian indicates that it is a local maximum for $T$ in the rest of the dotted region and a saddle with index and co-index~2 for $T$ in the yellow region. The transition from the pink to the yellow region is achieved across the curve
	\begin{equation}\label{transition}
		T_1(t)=\frac{4t^2(1-t)}{16t^4+1},
		\qquad
		T_2(t)=\frac{t(4t-1)}{16t^4+1},\qquad t\in\big(\tfrac14,1\big).
	\end{equation}
The Hessian at the critical points corresponding to this curve has two zero eigenvalues. However, we do not know whether these critical points lie on critical submanifolds as in Example~\ref{surface_critical}.

To verify the above claim about $S_{|\M_T}$ having a non-diagonal critical point for $T$ in the pink region, we first observe that $\Ric(g)$ is diagonal if and only if $x_3=x_4=0$ or $x_0^2=4x_1x_2$. This is because the off-diagonal entries in the Ricci curvature are 
$$
\frac{x_i(4x_1x_2-x_0^2)}{x_0(x_1x_2 - x_3^2 - x_4^2)}, \qquad i=1,2.
$$
Since the action of $N_0(H)/H$ leaves diagonal Ricci curvature unchanged, it suffices to consider only metrics $g$ with~$x_4=0$. We take two million non-diagonal~$g$ satisfying $x_0^2=4x_1x_2$ and $x_4=0$ and calculate $\Ric(g)$. After checking that $\Ric(g)$ is positive-definite and normalizing so that its $\fm_0$ component becomes~1, we mark the $\fm_1$ and $\fm_2$ components in the $(T_1,T_2)$-plane as in Figure~\ref{Stiefel_numeric}. The obtained values fill up the pink region. Note that the critical circles, which are obtained by applying the normalizer,  have radius $x_3$ with $0<x_3<\frac12$. As this radius becomes zero, we obtain the curve in~\eqref{transition}. Hence here the critical circles merge with the diagonal critical point. In the boundary on the right-hand side of the pink region, we necessarily have critical circles with at least two zero eigenvalues in the Hessian since, as one moves to the right, they must disappear.

	For a specific example, choose $T_0=1$, $T_1=\frac{135}{472}$ and $T_2=\frac{15}{118}$, the green dot in Figure~\ref{Stiefel_numeric}. Direct verification shows that the gradient of $S_{|\M_T}$ vanishes at the circle of metrics with 
	$x_0=\frac{181}{59}$, $x_1=\frac{905}{472}$, $x_2=\frac{362}{295}$ and $x_3^2+x_4^2=\frac{32761}{55696}$. These metrics are all saddles with index~2 and co-index~1. At the same time, the diagonal metric with
	$$
	x_0=\frac{\sqrt{20866}}{59},\qquad x_1=\frac{62598+277\sqrt{20866}}{54044}\qquad\mbox{and}\qquad x_2=\frac{83464-439\sqrt{20866}}{22892}
	$$
	is a strict local maximum. It cannot be a global maximum because its scalar curvature is less than~$\alpha_{\fk^0}$ and there are no further critical points.
	
	It follows that, outside the pink region, the functional has only a diagonal critical point. This must be the unique critical point. Thus in the dark-grey or middle-grey region in  Figure~\ref{Stiefel_regions} the global maximum must be a diagonal metric.
	\begin{figure}
		\centering
		\includegraphics[width=102mm]{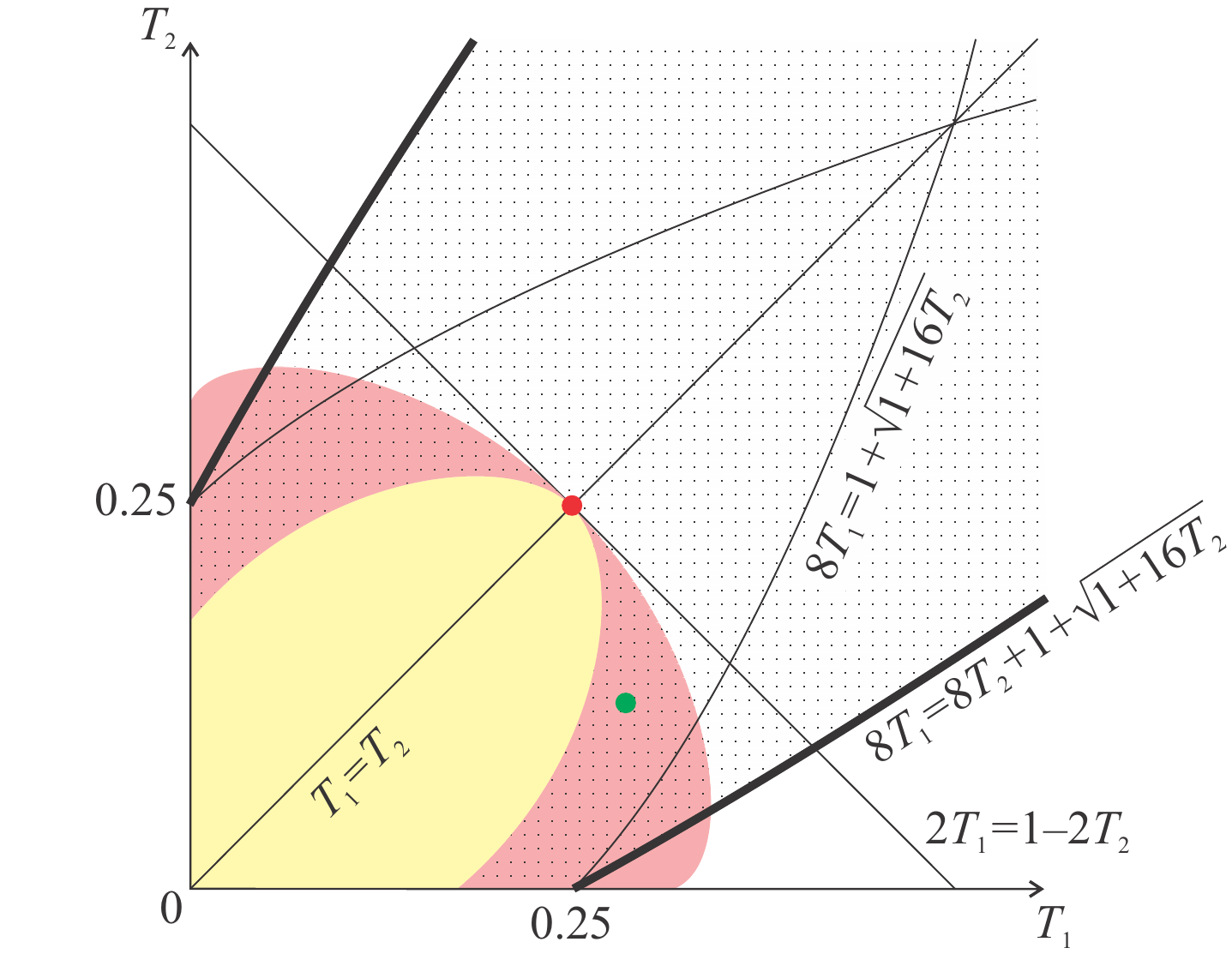}
		\caption{Ricci curvature of non-diagonal metrics on the Stiefel manifold $V_2(\R^4)$}
		\label{Stiefel_numeric}
	\end{figure}
\end{example}
	 As far as non-diagonal Ricci candidates $T$ are concerned, we make the following remarks. 
	Recall that $V_2(\R^4)$ supports two circles of (non-diagonal) Einstein metrics isometric to the canonical product Einstein metric on $V_2(\mathbb R^4)\simeq\Sph^3\times\Sph^2$; see~\cite{BWZ04}. Scaling each factor does not change the Ricci curvature and hence yields an arc of critical points. These arcs turn out to be non-degenerate critical submanifolds with index~3 and local maxima. Applying the normalizer, we obtain a circle of candidates $T$, each one of which admits  two arcs of critical points.
	
The Stiefel manifold is also a generalized Wallach space under the decomposition $D_{s,\theta}$ with $s=\frac1{\sqrt{2}}$ and any $\theta$. These spaces we will study in detail in~\cite{PZ23} in order to produce saddle critical points. Using the methods of~\cite{PZ23}, one shows that $S_{|\M_T}$ admits a critical point of co-index 0 or 1 if $$T_1=T_2<\frac14-\frac{2|T_3^2-T_4^2|}{\sqrt{T_3^2+T_4^2}}.$$
When $T_3=T_4=0$, this should be interpreted as $T_1=T_2<\frac14$, corresponding to the saddles discussed above in Example~\ref{saddle_stiefel}.

Finally, we remark that one can determine the metrics that are degenerate critical points. In \cite[Proposition~2.30]{PZ23} we show that $g$ is a non-degenerate critical point of $S_{|\M_T}$ if and only if $\rank d\Ric_g=\dim\M-1$. Assuming $x_0=1$ and $x_4=0$, a Maple computation shows that this only fails for three families of metrics. The first one satisfies $4x_1x_2=1$, and these are the metrics with diagonal Ricci tensor in Example~\ref{exa_max+saddle}. The second one is given by $x_3=\frac12$ with $x_1\ne x_2$, and for these the Ricci tensor is not diagonal. The third family consists of the metrics for which
$$
(16x_1^2x_3^2 + 4x_1^2 - 8x_3^2)x_2^2 + (-16x_1x_3^4 + 8x_1x_3^2 - x_1)x_2 + x_3^2 - 8x_1^2x_3^2 + 4x_3^4=0.
$$

\subsection{The Ledger--Obata space $\mathbf{\emph{H}^3/\text{diag}(\!\emph{H})}$ and $\mathbf{\emph{Spin}(8)/\emph{G}_2}$}

Our second example is motivated by the following observations. Let the homogeneous space $G/H$ be such that $\fm=\fm_1\oplus\fm_2$ with both summands irreducible. When $\fm_1$ and $\fm_2$ are inequivalent, one can solve the equation $\Ric(g)=cT$ directly; see~\cite{AP16}. There are only two cases where $\fm_1$ and $\fm_2$ are equivalent representations. If $G$ is simple, then $G/H$ must equal $Spin(8)/G_2=\Sph^7\times\Sph^7$; see the classification in~\cite{WDMK08,CH12}. If it is not, one easily checks that the only possibility is that $G/H$ is the Ledger--Obata space~$H^3/\diag(H)$, where $H$ is simple and embedded diagonally. 

We discuss here the latter case, the former one being quite similar; see Remark~\ref{rem_S7xS7}. In what follows, $(e_i)$ is a basis of~$\fh$ orthonormal with respect to~$-B_H$, the negative of the Killing form of~$H$.

Let $G/H=H^3/\diag(H)$. There are exactly three intermediate subgroups, namely,
\begin{align*}
	K_1&=\{(a,b,b)\in G\mid a,b\in H\},
	\\
	K_2&=\{(a,b,a)\in G\mid a,b\in H\}
	\qquad\mbox{and}\qquad
	K_3=\{(a,a,b)\in G\mid a,b\in H\}.
\end{align*}
The outer automorphism $R:G\to G$ given by $R(a,b,c)=(c,a,b)$ interchanges these subgroups. Choose the bi-invariant metric $Q$ on $G$ to be $Q=-B_H-B_H-B_H$. Then $R$ is an isometry of~$(G,Q)$.

Fix an $\Ad_H$-invariant, $Q$-orthogonal decomposition of $\fm$ by setting
\begin{equation}\label{decomp}
	\fm_1=\{(-2X,X,X)\in\fg\mid X\in\fh\}\qquad\text{and}\qquad \fm_2=\{(0,X,-X)\in\fg\mid X\in\fh\}. 
\end{equation}
Then $\fk_1=\fh\oplus\fm_1$. Furthermore, $\Ad_H$ is an irreducible representation of real type on each $\fm_i$ since $H$ is simple. The collections $\Big(\frac1{\sqrt6}(-2e_j,e_j,e_j)\Big)$ and $\Big(\frac1{\sqrt2}(0,e_j,-e_j)\Big)$ constitute $Q$-orthonormal bases of $\fm_1$ and $\fm_2$. With respect to these bases, the metric and $T$ have the form
\begin{equation}\label{LObata_metric}
	\centering
	g=
	\left(\begin{matrix}
		x_1&x_3\\
		x_3&x_2
	\end{matrix}\right) \qquad\text{and}\qquad
	T=\left(\begin{matrix}
		T_1&T_3\\
		T_3&T_2
	\end{matrix}\right)
\end{equation}
with $\det g>0$ and $\det T>0$. (Each entry in these matrices represents a scalar $a\times a$ matrix, where $a=\dim H$.) Applying our main theorem, we obtain the following result.
\begin{prop}\label{prop_LObata}
	Suppose that $G/H=H^3/\diag(H)$ and $T_1T_2>T_3^2$. The functional $S_{|\M_T}$ attains its global maximum if
	\begin{align}\label{LObata_cond}
		\frac34T_2<T_1<\frac95T_2-\frac{14\sqrt3}5|T_3|.
	\end{align}
\end{prop}

\begin{proof} We need to compute $\alpha_{\fk_i}$ and $\beta_{\fk_i}$. A well-known formula for the sectional curvature of a normal homogeneous metric (see~\cite[Proposition~7.87b]{AB87}) implies
	\begin{align*}
		S(Q_{|G/K_1})=\sum_{j,k}\sec(Q_{|G/K_1})\bigg(\frac{(0,e_j,-e_j)}{\sqrt2},\frac{(0,e_k,-e_k)}{\sqrt2}\bigg)=\frac12\sum_{j,k}|[e_j,e_k]|^2,
	\end{align*}
	where $|\cdot|$ is the norm corresponding to~$-B_H$. Similarly,
	\begin{align*}
		S(Q_{|K_1/H})=\sum_{j,k}\sec(Q_{|K_1/H})\bigg(\frac{(-2e_j,e_j,e_j)}{\sqrt6},\frac{(-2e_k,e_k,e_k)}{\sqrt6}\bigg)=\frac38\sum_{j,k}|[e_j,e_k]|^2.
	\end{align*}
	In addition, $(K_1/H,Q)$ is an isotropy irreducible symmetric space, and hence $[\fm_1,\fm_1]\subset\fh$. It follows from~\eqref{scalcurvx} that
	$$
	S(Q_{|G/K_1})=\frac a2 \qquad\text{and, by the above,}\qquad S(Q_{|K_1/H})=\frac{3a}8
	$$
	with $a=\dim H$. The trace constraints on $G/K_1$ and $K_1/H$ have the form $x_2=aT_2$ and $x_1=aT_1$, respectively. We conclude that
	\begin{equation*}
		\alpha_{\fk_1}=\frac{3}{8T_1}\qquad\text{and}\qquad\beta_{\fk_1}=\frac{1}{2T_2},
	\end{equation*}
	which means $\beta_{\fk_1}-\alpha_{\fk_1}>0$ if and only if $4T_1>3T_2$.
	
	As observed above, the automorphism $R$ takes $K_1$ to $K_2$. Since $R^3$ is the identity, the matrix of the pullback of $T$ by $R$ with respect to our fixed bases of~$\fm_1$ and $\fm_2$ is
	\begin{equation*}
		\left(\begin{matrix}
			-\frac12 &\frac{\sqrt{3}}{2}\\
			-\frac{\sqrt{3}}{2}&-\frac12
		\end{matrix}
		\right)
		\left(\begin{matrix}
			T_1&T_3\\
			T_3&T_2
		\end{matrix}
		\right)
		\left(\begin{matrix}
			-\frac12 &-\frac{\sqrt{3}}{2}\\
			\frac{\sqrt{3}}{2}&-\frac12
		\end{matrix}
		\right)
		=
		\left(\begin{matrix}
			\frac{T_1+3T_2-2\sqrt{3}T_3}{4} &\frac{\sqrt{3}(T_1-T_2)-2T_3}{4}\\
			\frac{\sqrt{3}(T_1-T_2)-2T_3}{4}&\frac{3T_1+T_2+2\sqrt3T_3}{4}
		\end{matrix}
		\right).
	\end{equation*}
	The maps from $(G/K_1,Q)$ to $(G/K_2,Q)$ and from $(K_1/H,Q)$ to $(K_2/H,Q)$ induced by $R$ are isometries. They preserve the scalar curvature, so
	\begin{equation*}
		\alpha_{\fk_2}=\frac{3}{2(T_1+3T_2-2\sqrt{3}T_3)}\qquad\text{and}\qquad
		\beta_{\fk_2}=\frac{2}{3T_1+T_2+2\sqrt{3}T_3}.
	\end{equation*}
	Consequently, $\beta_{\fk_2}-\alpha_{\fk_2}>0$ if and only if $5T_1<9T_2-14\sqrt{3}T_3$.
	
	To obtain the third subgroup, we need to apply $R$ to $K_2$, or $R^2$ to~$K_1$. The matrix of the pullback of $T$ by $R^2$ with respect to~\eqref{decomp} is
	\begin{equation*}
		\left(\begin{matrix}
			\frac{T_1+3T_2+2\sqrt{3}T_3}{4} &\frac{\sqrt{3}(T_2-T_1)-2T_3}{4}\\
			\frac{\sqrt{3}(T_2-T_1)-2T_3}{4}&\frac{3T_1+T_2-2\sqrt3T_3}{4}
		\end{matrix}\right).
	\end{equation*}
	This implies
	\begin{equation*}
		\alpha_{\fk_3}=\frac{3}{2(T_1+3T_2+2\sqrt{3}T_3)}\qquad\text{and}\qquad\beta_{\fk_3}=\frac{2}{3T_1+T_2-2\sqrt{3}T_3}.
	\end{equation*}
	Thus $\beta_{\fk_3}-\alpha_{\fk_3}>0$ if and only if $5T_1<9T_2+14\sqrt{3}T_3$.
	
	There are two scenarios for the condition of our main theorem to be satisfied. One is that $\alpha_{\fk_1}\ge\max\{\alpha_{\fk_2},\alpha_{\fk_3}\}$ (equivalently, $3T_1\le3T_2-2\sqrt{3}|T_3|$), in which case it suffices to demand that ${4T_1>3T_2}$. The other is that $\alpha_{\fk_1}\le\max\{\alpha_{\fk_2},\alpha_{\fk_3}\}$. Then we need $5T_1<9T_2-14\sqrt{3}|T_3|$. Combining these conditions, we arrive at~\eqref{LObata_cond}.
\end{proof}

For $g$ and $T$ of the form~\eqref{LObata_metric}, the constraint $\tr_gT=1$ is
\begin{equation*}
	\frac{a(x_1T_2+x_2T_1-2x_3T_3)}{x_1x_2-x_3^2}=1,
\end{equation*}
and a computation shows that the scalar curvature is given by
\begin{equation}\label{LObata_scal_curv}
	S(g)=\frac{a(9x_1x_2^2+12x_1^2x_2-6x_1x_3^2-18x_2x_3^2-x_1^3)}{24(x_1x_2-x_3^2)^2}.
\end{equation}
Thus we have the same critical points, up to scaling, no matter what group $H$ we choose. The formulas
$$
x=\frac12(T_1-T_2)\qquad\mbox{and}\qquad y=T_3
$$
define coordinates in the space of Ricci candidates $T$ normalized so that $T_1+T_2=1$. Figure~\ref{LObata_T_sym} shows points in the $(x,y)$-plane that correspond to different behaviors of~$S_{|\M_T}$. The tensor $T$ is positive-definite if and only if $x^2+y^2<\frac14$. The order-three automorphism $R$ yields a natural symmetry on the space of Ricci candidates. Its action rotates the picture in Figure~\ref{LObata_T_sym} by~$\frac{2\pi}3$. The inside of the large triangle is the set of $(x,y)$ that satisfy the condition of Proposition~\ref{prop_LObata}. The equalities $\alpha_{\fk_1}=\alpha_{G/H}$, $\alpha_{\fk_2}=\alpha_{G/H}$ and $\alpha_{\fk_3}=\alpha_{G/H}$ hold in the dark-grey, the middle grey and the light-grey triangle, respectively.  A computer-assisted experiment with one million metrics shows that $S_{|\M_T}$ also has a critical point for $(x,y)$ in the pink regions. Indefinite tensors that are Ricci curvatures of metrics up to scaling fill up the blue regions. Thus the image of the Ricci map is the union of the grey, pink and blue areas. In particular, for any $T$ in the white region $S_{|\M_T}$ has no critical points.

Maple is able to solve the Euler--Lagrange equations for $S_{|\M_T}$ for any specific choice of~$T$. It suggests that the solution representing a metric, when it exists, is unique and is always a local maximum (global if $T$ lies inside the triangle).

\begin{figure}
	\centering
	\includegraphics[width=95mm]{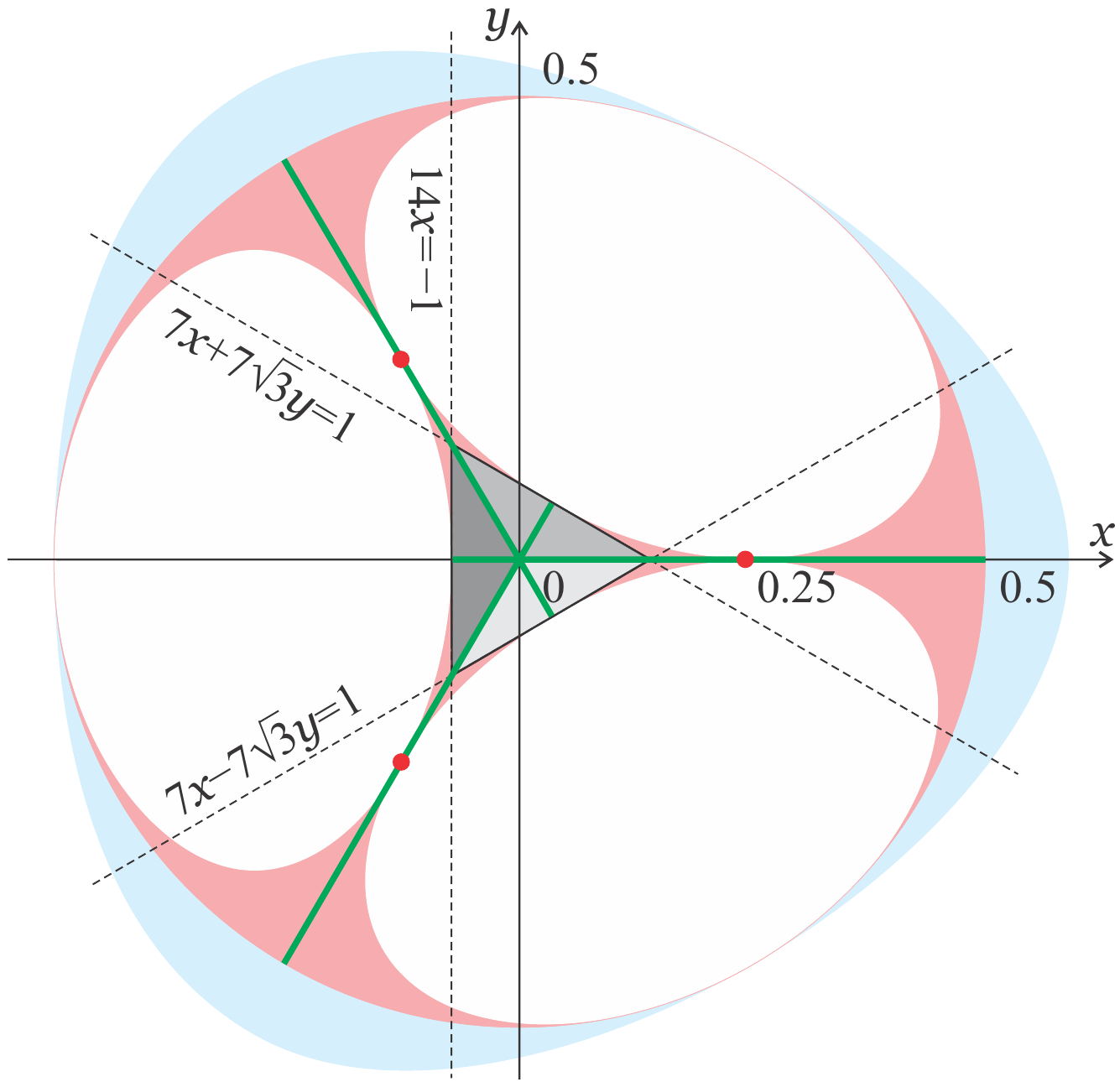}
	\caption{Prescribed Ricci curvature on the Ledger--Obata space $H^3/\diag(H)$}
	\label{LObata_T_sym}
\end{figure}

\begin{example}\label{Obata_diagonal}
	Assume that $T_1=t$, $T_2=1-t$ and $T_3=0$. One easily checks that the formulas
	\begin{align}\label{y=0}
		x_1= a\sqrt{-20t^2 + 30t - 9},\quad
		x_2=a\frac{-20t^2+ 30t - 9 + t\sqrt{-20t^2 + 30t - 9} }{3(7t - 3)}\quad\text{and}
		\qquad x_3=0
	\end{align}
	define a metric with positive Ricci curvature for $\frac37<t<1$, and this metric is a critical point of~$S_{|\M_T}$. Computing the Hessian, we conclude that it is a strict local maximum unless~$t=\frac34$. In Figure~\ref{LObata_T_sym} the Ricci candidates with $T_1=t\in\big(\frac37,1\big)$, $T_2=1-t$ and $T_3=0$ occupy the segment of the $x$-axis with $x\in\big(\!-\frac1{14}, \frac12\big)$. The orbit of this segment under $R$, indicated by the green lines, gives two further segments with the same behavior.

	The case of $t=\frac34$, respectively $x=\frac14$, is special. The orbit of the corresponding tensor $T$ under $R$ is depicted by the red dots in Figure~\ref{LObata_T_sym}. The three tensors in this orbit are, in fact, Einstein metrics. The $H^3$-equivariant diffeomorphisms from $H^3/\diag(H)$ to $H\times H$ given by 
	\begin{align*}
		(a,b,c)\diag(H)&\mapsto (ac^{-1},bc^{-1}),
		\\ (a,b,c)\diag(H)&\mapsto (ab^{-1},cb^{-1})\qquad\mbox{and}\qquad(a,b,c)\diag(H)\mapsto (ba^{-1},ca^{-1})
	\end{align*}
	are isometries between these Einstein metrics and the canonical Einstein product metric on~$H\times H$. Scaling the factors in the product metric, we obtain arcs of critical points.  As it turns out, these arcs are non-degenerate critical submanifolds with index~1 and co-index~$0$ and hence consist of local maxima. 
	
	The only other Einstein metric on $H^3/\diag(H)$ is the normal homogeneous metric induced by~$Q$; see~\cite{CNN17}. It corresponds to the origin in Figure~\ref{LObata_T_sym} and is a strict local maximum of the associated functional~$S_{|\M_T}$. 
	
	As in the case of the Stiefel manifold, one easily determines which critical points are degenerate. It turns out that this only happens for the ``red dot" Einstein metrics discussed above. Thus for the Ledger--Obata space all the functionals  $S_{|\M_T}$ are Morse or Morse--Bott functions.
\end{example}

\begin{rem}\label{rem_S7xS7}
The remaining homogeneous space with two equivalent summands is $Spin(8)/G_2$. It is, in fact, quite similar since the scalar curvature $S$ is, remarkably, given by the same formula as~\eqref{LObata_scal_curv} with $a=84$; see~\cite{Ke98}.   There are three intermediate subgroups isomorphic to $Spin(7)$, which are permuted by the triality automorphism of $Spin(8)$.  One  easily sees that the conditions for a global maximum on $Spin(8)/G_2$ are the same as on the Ledger--Obata space. The set of Einstein metrics consists of three metrics isometric to the canonical product Einstein metric on $\Sph^7\times\Sph^7$, and the normal homogeneous metric. 
\end{rem}

\bigskip

\providecommand{\bysame}{\leavevmode\hbox
	to3em{\hrulefill}\thinspace}

\end{document}